\documentclass[final,onefignum,onetabnum]{siamonline220329}

\usepackage{lipsum}
\usepackage{amsfonts}
\usepackage{graphicx}
\usepackage{epstopdf}
\usepackage{algorithmic}

\usepackage{amssymb}

\ifpdf
  \DeclareGraphicsExtensions{.eps,.pdf,.png,.jpg}
\else
  \DeclareGraphicsExtensions{.eps}
\fi

\usepackage{enumitem}
\setlist[enumerate]{leftmargin=.5in}
\setlist[itemize]{leftmargin=.5in}

\newsiamremark{remark}{Remark}
\newsiamremark{hypothesis}{Hypothesis}
\crefname{hypothesis}{Hypothesis}{Hypotheses}
\newsiamthm{claim}{Claim}

\newtheorem{assumption}[theorem]{Assumption}

\newtheorem{example}[theorem]{Example}

\newcommand{\mcG}{\mathcal{G}}

\newcommand{\Id}{\mathrm{Id}}

\providecommand{\mathbbm}{\mathbb} 

\newcommand{\R}{\mathbbm{R}}
\newcommand{\E}{\mathbbm{E}}

\newcommand{\N}{\mathbbm{N}}

\renewcommand{\L}{\mathcal{L}}

\newcommand{\D}{\mathcal{D}}

\definecolor{mygreen}{rgb}{0.13,0.55,0.13}

\newcommand{\nc}{\normalcolor}

\newcommand{\iid}{\stackrel{\text{i.i.d.}}{\sim}}

\newcommand{\Nc}{\mathcal{N}}

\usepackage{comment}

\headers{Precision and Cholesky factor estimation}{J.Chen and D.Sanz-Alonso}

\title{Precision and Cholesky Factor Estimation for Gaussian Processes}

\author{Jiaheng Chen\thanks{Committee on Computational and Applied Mathematics, University of Chicago, Chicago, IL 
  (\email{jiaheng@uchicago.edu}).}
\and Daniel Sanz-Alonso\thanks{Department of Statistics, University of Chicago, Chicago, IL (\email{sanzalonso@uchicago.edu}).}}

\usepackage{amsopn}
\DeclareMathOperator{\diag}{diag}

\ifpdf
\hypersetup{
  pdftitle={Precision operator estimation},
  pdfauthor={J. Chen and D. Sanz-Alonso}
}
\fi

\begin{document}

\maketitle

\begin{abstract}
This paper studies the estimation of large precision matrices and Cholesky factors obtained by observing a Gaussian process at many locations. Under general assumptions on the precision and the observations, we show that the sample complexity scales poly-logarithmically with the size of the precision matrix and its Cholesky factor. The key challenge in these estimation tasks is the polynomial growth of the condition number of the target matrices with their size. For precision estimation, our theory hinges on an intuitive local regression technique on the lattice graph which exploits the approximate sparsity implied by the screening effect. For Cholesky factor estimation, we leverage a block-Cholesky decomposition recently used to establish complexity bounds for sparse Cholesky factorization. 
\end{abstract}



\section{Introduction}\label{sec:introduction}
This paper studies nonparametric estimation of large precision matrices obtained by observing a Gaussian process at many scattered locations. Under general assumptions on the precision and the observations, we show that the sample complexity scales poly-logarithmically with the size of the precision matrix. The key challenge in this estimation task is the polynomial growth of the precision's condition number with its size. Our theory hinges on an intuitive local regression technique on the lattice graph to exploit the approximate sparsity implied by the \emph{screening effect} \cite{stein2002screening}, along with a novel reduction argument based on Hall's marriage theorem \cite{hall1987representatives} to lift the results from the lattice graph to scattered observation locations. In addition, we study the estimation of the precision's upper-triangular Cholesky factor building on a block-Cholesky decomposition recently used to establish complexity bounds for sparse Cholesky factorization \cite{schäfer2021compression,schäfer2021sparse}. 

Although precision estimation under structural assumptions is a central problem in statistics \cite{friedman2008sparse,cai2011constrained,ren2015asymptotic,cai2016estimating,cai2016estimatingEJS,hu2017minimax,liu2020minimax}, the question of estimating \emph{ill-conditioned} precision matrices remains largely unexplored. This limitation of the existing theory is surprising, given that precision matrices arising from graphical models on a path, Brownian motion, Brownian bridge, Mat\'ern processes, and many other natural models are ill-conditioned in high dimension \cite{kelner2020learning,kelner2022power}. In addition, \cite{kelner2020learning} points out that well-conditioning is a fragile assumption, since it fails even under benign operations such as rescaling the variables. Consistent estimation generally requires that the sample size scales \emph{linearly} with the condition number of the precision matrix \cite[Section 5.2]{ren2015asymptotic}. In our setting, where the condition number grows polynomially, results as those in \cite{ren2015asymptotic} would imply the need for impractical sample sizes. Fortunately, precision matrices arising from observing Gaussian processes at scattered locations naturally satisfy approximate sparsity due to the screening effect, which enables us to show that the sample complexity scales \emph{poly-logarithmically} with the size of the precision matrix. Thus, our theory reveals an interesting tension between the challenge of ill-conditioning and the benefit of approximate sparsity.  

Motivated by applications in sampling and transport where it is often preferable to work with the precision's Cholesky factor rather than the precision itself \cite{rue2005gaussian,marzouk2016sampling,kidd2022bayesian,katzfuss2024scalable}, this paper additionally studies Cholesky factor estimation. 
Here, the issue of ill-conditioning plays again an important role. Perturbation bounds for the Cholesky factorization generally depend on the condition number of the matrices \cite{drmavc1994perturbation,edelman1995parlett}. Therefore, the naive approach of using the upper-triangular Cholesky factor of the estimated precision is inadequate unless additional structure of the precision is exploited. We will show that operator-adapted wavelets (also known as \emph{gamblets}) \cite{owhadi2017multigrid,owhadi2017universal,owhadi2019operator,owhadi2019statistical} provide a natural framework to address this challenge.

\subsection{Related work} 
To our knowledge, only a few studies have explored precision learning and estimation without imposing condition number bounds \cite{misra2020information,kelner2020learning}. These works focus primarily on recovering the graph structure of Gaussian graphical models (GGMs), rather than directly estimating the precision. Misra et al. \cite{misra2020information} derived the information-theoretic lower bound on the sample size required to recover the dependency graph in sparse GGMs. Kelner et al. \cite{kelner2020learning} proposed polynomial-time algorithms for learning \emph{attractive} GGMs and \emph{walk-summable} GGMs with a logarithmic number of samples. In contrast, our estimation task extends beyond these two classes and focuses on addressing the challenges in the estimation of ill-conditioned precision matrices. In our problem setting, formalized in Section \ref{sec:mainresults}, the condition number of the precision matrix grows polynomially with its size.  Quantitative bounds on the largest and smallest eigenvalues of the precision were derived in \cite[Lemma 15.43]{owhadi2019operator} using a Poincar\'e inequality closely related to the accuracy of numerical homogenization basis functions \cite{maalqvist2014localization,owhadi2014polyharmonic,hou2017sparse} and an inverse Sobolev inequality related to the regularity of the discretization of the precision operator. We refer to Lemma \ref{lemma:magnitude} in this paper for a formal statement and to \cite{owhadi2019operator,schäfer2021compression}
for proofs and additional references.

An important feature of precision matrices arising from observing Gaussian processes is their approximate sparsity, which stems from the \emph{screening effect} ---an important phenomenon in spatial statistics by which the value of a process at a given location becomes nearly independent of values at distant locations when conditioned on values at nearby locations \cite{stein2002screening,stein20112010,cressie2015statistics}. 
It is well known that the zero patterns of the precision matrix encode the conditional independence structure due to the Markov property \cite[Theorem 2.2]{rue2005gaussian}.
Similarly, \emph{approximate conditional independence} for smooth Gaussian processes yields \emph{approximate sparsity} of their precision matrices. Quantitative results on the exponential decay of the conditional correlation have been established in \cite{owhadi2017multigrid,owhadi2017universal,owhadi2019operator}; see Lemma \ref{lemma:sparsity} in this paper for a formal statement.

As mentioned above, our study of Cholesky factor estimation builds on the framework of operator-adapted wavelets \cite{owhadi2017multigrid,owhadi2017universal,owhadi2019operator,owhadi2019statistical}.
 This line of work identifies wavelets adapted to the precision operator, in the sense that the matrix representation of the precision in the basis formed by these wavelets is block-diagonal with uniformly well-conditioned and sparse blocks. These wavelets exhibit three key properties: (i) orthogonality across scales in the energy scalar product; (ii) uniform boundedness of the condition numbers of the operator within each subband (i.e., uniform Riesz stability in the energy norm); and (iii) exponential decay.
As noted in \cite{schäfer2021compression}, Cholesky factorization under the maximin ordering of observation locations formalized in Assumption \ref{assumption:ordering} below is intimately related to computing operator-adapted wavelets. For details, we refer the reader to \cite[Section 3.3]{schäfer2021compression}. Additional references include the expository paper \cite{owhadi2019statistical} and the monograph \cite{owhadi2019operator}. Leveraging this connection, \cite{schäfer2021compression} established a block-Cholesky factorization formula (see Lemma \ref{lemma:decomposition} in this paper for a statement), provided quantitative estimates on the spectrum of the stiffness matrices $B^{(k)}$ on the diagonal (see Lemma \ref{lemma:bounded_cond}), and demonstrated the accuracy of approximations in sparse Cholesky factorization. Building on these rich structures and theoretical foundations, we focus, from a statistical perspective, on estimating the Cholesky factor of the precision under the maximin ordering.

\subsection{Outline and main contributions}
\begin{itemize}
\item Section \ref{sec:mainresults} presents and discusses the main results of this paper. Theorem \ref{thm:mainresult1}, Theorem \ref{thm:mainresult2}, and Theorem \ref{thm:mainresult3} establish high-probability error bounds for three key problems: precision estimation for Gaussian processes with homogeneously scattered observations, Cholesky factor estimation under the maximin ordering of observation locations, and precision operator estimation on the lattice graph. These results demonstrate that, under general assumptions on the precision operator and observation locations, the sample complexity scales poly-logarithmically with the size of the system. For precision estimation on the lattice graph, Theorem \ref{thm:mainresult3} bounds the error of an estimator explicitly defined in Subsection \ref{ssec:lattice}. In contrast, Theorems \ref{thm:mainresult1} and \ref{thm:mainresult2} establish the existence of an estimator with the aforementioned sample complexity.

\item Section \ref{sec:proof_smooth} contains the proof of Theorem \ref{thm:mainresult1}. Our proof technique leverages that scattered observation locations are close to points on a slightly larger regular lattice. This intuitive insight, which we formalize using Hall's marriage theorem, allows us to relate the sample complexity of precision estimation with scattered observation locations to that of precision estimation on a lattice graph. Consequently, while the lattice setting in Theorem \ref{thm:mainresult3} facilitates the description and analysis of our local estimation procedure, it ultimately suffices to also capture the sample complexity of the problem with scattered observations in Theorem \ref{thm:mainresult1}. Our reduction argument based on Hall's marriage theorem may be of independent interest in analyzing other problems involving homogeneously scattered points in a physical domain. 
\item Section \ref{sec:proof_Cholesky} is devoted to the proof of Theorem \ref{thm:mainresult2}, which builds upon a block-Cholesky decomposition recently used to establish complexity bounds for sparse Cholesky factorization \cite{schäfer2021compression}. A central idea in the analysis is to integrate the information across all scales within the hierarchical structure induced by the maximin ordering of observation locations.
\item Section \ref{sec:proof_lattice} investigates precision operator estimation on the lattice graph and proves Theorem \ref{thm:mainresult3}. A key focus is on establishing the explicit dependence of the estimation error on the condition number, which is essential in our proof of Theorem \ref{thm:mainresult1}.

\item Section \ref{sec:Conclusions} closes with conclusions, discussion, and future directions.
\end{itemize}

\subsection{Notation}
 For positive integer $s,$ we denote by $H_0^s(\D)$ the Sobolev space of functions in $\D=[0,1]^d$ with order $s$ derivatives in $L^2$ and zero traces, and by $H^{-s}(\D)$ the dual space of $H_0^s(\D)$. We assume $s>d/2$ throughout the paper. Then, Sobolev's embedding theorem ensures that $H_0^s(\D)\subset C(\D)$, where $C(\D)$ is the space of continuous functions on $\D$. Since in most scientific applications the dimension $d$ of the physical space is not large (typically $d\le 3$), we treat it as a constant in our analysis. For positive sequences $\{a_n\}, \{b_n\}$, we write $a_n \lesssim b_n$ to denote that, for some constant $c >0$, $a_n \le c b_n.$ If both $a_n \lesssim b_n$ and $b_n \lesssim a_n$ hold, we write $a_n \asymp b_n$. For real numbers $a$ and $b$, we use the notation $a\lor b:=\max\{a,b\}$. For vector $x\in\R^m$, we denote by $\|x\|$ its Euclidean norm. For matrix $A \in \R^{m_1\times m_2},$ we denote by $\|A\|$ its spectral norm, and for nonsingular $A \in \R^{m \times m}$ we denote by $\kappa(A)=\|A\|\|A^{-1}\|$ its condition number.
 
\section{Main results}\label{sec:mainresults}
This section states and discusses the main results of the paper. In Subsection \ref{ssec:precision_operator}, we study precision estimation for Gaussian processes and establish a high-probability estimation error bound in Theorem \ref{thm:mainresult1}. Subsection \ref{ssec:Cholesky} focuses on estimating the upper Cholesky factor of the precision, with Theorem \ref{thm:mainresult2} providing a high-probability error bound. Finally, in Subsection \ref{ssec:lattice} we study precision operator estimation on the lattice graph and show in Theorem \ref{thm:mainresult3} an error bound with an explicit dependence on the condition number; this result is of independent interest and also facilitates the proof of Theorem \ref{thm:mainresult1}.

\subsection{Precision estimation for Gaussian processes}\label{ssec:precision_operator}

Let $\{u_n\}_{n=1}^N$ be $N$ independent copies of a real-valued, centered, continuous Gaussian process on $\D=[0,1]^d$ with precision operator $\L$. We observe data $\{Z_n\}_{n=1}^N$ consisting of local measurements of the random functions $\{u_n\}_{n=1}^N \iid \Nc(0, \L^{-1})$ at $M$ distinct observation locations $\{x_i\}_{i = 1}^M \subset \D \, $:
\[
Z_n:= \bigl(u_n(x_1),u_n(x_2),\ldots u_n(x_M) \bigr)^{\top}\in \R^M,\quad 1\le n\le N.
\]
Consider the covariance and precision matrices given by 
\begin{equation}\label{eq:precisiondef}
    \Sigma :=\E [Z_nZ_n^{\top}]= \bigl(G(x_i,x_j)\bigr)_{1\le i,j\le M}, \qquad \Omega:=\Sigma^{-1},
\end{equation}
where $G(\cdot,\cdot)$ is the covariance function of the process, given by the Green's function of $\L$. Our goal is to estimate $\Omega$ from the data $\{Z_n\}_{n=1}^N.$

Motivated by recent work on Gaussian process regression, numerical approximation, and elliptic PDE solvers \cite{owhadi2019operator}, we consider the class of Gaussian processes whose precision operator $\L$ satisfies the following general assumption \cite[Section 2.2]{owhadi2019operator}. Here, we denote by $[f,v]$ the duality pairing between $f \in H^{-s}(\mathcal{D})$ and $v \in H_0^s(\mathcal{D});$ we refer to \cite{owhadi2019operator} for a formal definition.

\begin{assumption}\label{assumption:operator} Let $s>d / 2.$ The following conditions hold for $\L: H_0^s(\D)\to H^{-s}(\D)$: 
\begin{enumerate}[label=(\roman*)]
\item Symmetry: $[\L u,v]=[u,\L v]$ for all $u,v \in H_0^s(\D); $ 
\item Positive definiteness: $[\L u,u]\ge 0$ for $u\in H_0^s(\D)$;
\item Boundedness: 
\[
\|\L\|:=\sup_{u\in H_0^s(\D)} \frac{\|\L u\|_{H^{-s}(\D)}}{\|u\|_{H_0^s(\D)}}<\infty, \quad \|\L^{-1}\|:=\sup_{u\in H_0^s(\D)}\frac{\|u\|_{H_0^s(\D)}}{\|\L u\|_{H^{-s}(\D)}}<\infty;
\]
\item Locality: $[u,\L v]=0$ if $u \in H_0^s(\D)$ and $v \in H_0^s(\D)$ have disjoint supports in $\D$. \nc
\end{enumerate}
\end{assumption}

As noted in \cite{owhadi2019operator,schäfer2021compression}, a simple example of a precision operator satisfying Assumption \ref{assumption:operator} is $\L=(-\Delta)^s, s\in \N,$ $s> d/2$; here we use the zero Dirichlet boundary condition to define $\L^{-1}$. Assumption \ref{assumption:operator} also includes a broad class of precision operators of Gaussian processes widely used in practice, such as Mat\'ern kernels with the smoothness parameter satisfying $s=\nu+d/2\in \N$ for $\nu>0$ \cite{whittle1954stationary,lindgren2011explicit,stein2012interpolation}. 
Since we assume throughout that $s>d/2$, Sobolev's embedding theorem ensures that the Green’s function $G$ of $\L$ (with zero Dirichlet boundary condition) is a well-defined, continuous, symmetric, and positive definite kernel. Thus, we can consider centered Gaussian processes with covariance function $G$ \cite{owhadi2019operator}, and their covariance matrix $\Sigma= \bigl(G(x_i,x_j)\bigr)_{1\le i,j\le M}$ is invertible since the observation locations $\{x_i\}_{i=1}^M$ are assumed to be distinct. We will further assume the following: 

\begin{assumption}\label{assumption:data}
The observation locations $\left\{x_i\right\}_{i =1}^M\subset \D$ are homogeneously scattered in the sense that for $\mathsf{h}, \delta \in(0,1),$ the following conditions hold:
    
(1) $\sup _{x \in \D} \min _{1 \le i \le M}\|x-x_i\| \leq \mathsf{h}$;

(2) $\min _{1 \le i \le M} \inf _{x \in \partial \D}\|x-x_i\| \geq \delta \mathsf{h}$; 

(3) $\min_{1 \le i< j \le  M}\|x_i-x_j\| \geq \delta \mathsf{h}$.
\end{assumption}

The parameter $\delta>0$ quantifies the homogeneity of the point cloud $\{x_i\}_{i=1}^M$ and is treated as a fixed constant throughout the paper. An alternative way to define $\delta$ is (see e.g. \cite{schäfer2021compression,chen2024sparse})
\[
\delta:=\frac{\min _{i \neq j} \operatorname{dist}\left(x_i,\left\{x_j\right\} \cup \partial \D\right)}{\sup_{x \in \D} \operatorname{dist}\big(x,\left\{x_i\right\}_{i=1}^{M} \cup \partial \D\big)}.
\]
The parameter $\mathsf{h}$ represents the observation mesh-size, which indicates the resolution of our discrete observations of the underlying continuous Gaussian processes. A volume argument shows that 
$
M\asymp \mathsf{h}^{-d},
$
see e.g. \cite[Proposition 4.3]{owhadi2019operator}. As will be discussed below, Assumptions \ref{assumption:operator} and \ref{assumption:data} guarantee approximate sparsity of the precision $\Omega = \Sigma^{-1},$ which facilitates its estimation. However, these assumptions also imply that $\Omega$ is ill-conditioned as $\mathsf{h}\to 0$, posing a significant and largely unexplored challenge.

Our first main result establishes a high-probability, spectral-norm bound on the estimation error in terms of the sample size $N$ and the observation mesh-size $\mathsf{h}$. We prove Theorem \ref{thm:mainresult1} in Section \ref{sec:proof_smooth}.

\begin{theorem}\label{thm:mainresult1}
    Suppose that Assumptions \ref{assumption:operator} and \ref{assumption:data} hold. For any $r>0$, there are three constants $C_1, C_2, C_3>0$ depending only on $r,\delta,d,s, \|\L\|, \|\L^{-1}\|$ such that if
\[
N\ge C_1 \log^d (N/\mathsf{h}),
\]
then there exists an estimator $\widehat{\Omega}=\widehat{\Omega}(Z_1,Z_2,\ldots,Z_N)$ which satisfies with probability at least $1-C_2(\mathsf{h}\log (N/\mathsf{h}))^{dr}\exp(-(r+1)(\log N)^d)$ that
\[
\frac{\|\widehat{\Omega}-\Omega\|}{\|\Omega\|}\le C_3 \bigg(\frac{\log^d (N/\mathsf{h})}{N}\bigg)^{1/2}.
\]
\end{theorem}

To the best of our knowledge, Theorem \ref{thm:mainresult1} is the first result in the literature to consider precision estimation of Gaussian processes under the general Assumption \ref{assumption:operator} on the precision operator, which includes many Gaussian process models used in practice \cite{williams2006gaussian,lindgren2011explicit,stein2012interpolation}. Under general assumptions, Theorem \ref{thm:mainresult1} shows that one can construct an efficient estimator which only needs $N\gtrsim \log^{d}(1/\mathsf{h})\asymp (\log M)^d$ samples to control the relative error under the spectral norm. This contrasts with the fact that precision matrix estimation typically requires the sample size to grow \emph{linearly} with the matrix size to achieve a small error, i.e. $N\gtrsim M\asymp \mathsf{h}^{-d}$, since the sample covariance matrix is not invertible if $N<M\asymp \mathsf{h}^{-d}$. Such exponential improvement of the sample complexity becomes crucial as the observation mesh-size $\mathsf{h}\to 0$, which is the challenging regime in practice.

\begin{remark}

While the ill-conditioning of $\Omega$ poses a significant challenge for our estimation, it will be shown in Lemma \ref{lemma:sparsity} that $\Omega$ satisfies approximate sparsity due to the screening effect \cite{stein2002screening,stein20112010}. Among the various forms of sparsity studied in the literature, the one most closely related to our set-up is the \emph{banded} assumption proposed and analyzed in \cite{bickel2008regularized,cai2010optimal,hu2017minimax,al2024optimal}. However, we emphasize that precision matrices for Gaussian processes in spatial dimension $d \ge 2$ are, in general, not bandable, setting apart our work apart from existing high-dimensional results for precision estimation. Specifically, there does not exist any permutation matrix $P$ such that $P^{\top}\Omega P$ becomes a banded matrix with bandwidth at most $\mathsf{polylog}(M)$. We provide a heuristic argument for this claim. Recall that the \emph{dependency graph} of $\mathcal{N}(0,\Sigma)$ is defined as the graph whose adjacency matrix corresponds to the support of $\Omega=\Sigma^{-1}$ \cite{rue2005gaussian}. If $P^{\top}\Omega P$ were a banded matrix with bandwidth $\mathsf{polylog}(M)$, the dependency graph would necessarily have a separator of size $\mathsf{polylog}(M)$, dividing the vertices into two subsets, each containing at least $\mathcal{O}(M)$ vertices. However, this contradicts the properties of precision matrices $\Omega$ in our setting. Due to the locality of $\L$ in the physical space, the dependency graph of $\Omega$ resembles a grid graph, which has a large \emph{treewidth} of order $\mathcal{O}(M^{(d-1)/d})$ \cite{kelner2022power}. Consequently, estimating $\Omega$ in high-dimensional physical spaces ($d>1$) cannot be reduced to the estimation of a bandable precision matrix simply by reordering the points. Therefore, the minimax rates derived for estimating bandable precision matrices in \cite{hu2017minimax} are not directly applicable to our context.  

In light of the dependency graph structure of $\Omega$, the proof of Theorem \ref{thm:mainresult1} relies on precision operator estimation on the lattice graph developed in Subsection \ref{ssec:lattice} and Section \ref{sec:proof_lattice}. Moreover, as will become clear in Subsection \ref{ssec:lattice}, our estimation procedure circumvents the technically complicated tapering methods commonly used for estimating banded covariance or precision matrices \cite{cai2010optimal,hu2017minimax,liu2020minimax,al2024optimal}. Additionally, it provides a clearer interpretation through its connection to local regression; see Remark \ref{remark:GGM} for further details.

\begin{remark} In the context of Assumption \ref{assumption:operator} and Theorem \ref{thm:mainresult1}, we impose a zero Dirichlet boundary condition to define the covariance operator $\L^{-1}$, which requires working in $H_0^s(\D)$ rather than $H^s(\D)$. This corresponds to the
Mat\'ern covariance under conditioning on boundary values, a standard setup in the literature \cite{owhadi2019operator, schäfer2021compression,schäfer2021sparse,kang2024asymptotic}. However, as discussed in \cite{kang2024asymptotic}, the validity of this boundary conditioning assumption remains an open question in various applications, including geostatistics and computer-model emulation. For instance, as noted in \cite{schäfer2021compression}, when $\D=\R^d$, the Mat\'ern family of kernels is widely used and is equivalent to employing the whole-space Green's function of an elliptic PDE as the covariance function. Let $\bar{\D}$ be a bounded domain containing the measurement points $\{x_i\}_{i=1}^{M}$. In this case, the screening effect weakens near the boundary of $\bar{\D}$ due to the absence of measurement points outside $\bar{\D}$, leading to stronger conditional correlations between distant points near the boundary compared to similarly distant points in the interior. For a more detailed discussion on the impact of missing boundary conditioning in practical Mat\'ern models, as well as numerical verification of the screening effect, we refer the reader to \cite[Section 4.2]{schäfer2021compression}, \cite[Section 3.3.3]{schäfer2021sparse}, and \cite[Section 3.2]{kang2024asymptotic}.

\end{remark}

\end{remark}

\subsection{Cholesky factor estimation}\label{ssec:Cholesky}

Our second main result, Theorem \ref{thm:mainresult2} below, shows that under the maximum minimum distance (maximin) ordering \cite{guinness2018permutation,schäfer2021compression} of the observation locations outlined in Assumption \ref{assumption:ordering}, there exists an estimator for the upper-triangular Cholesky factor of $\Omega$ which achieves nearly optimal convergence rate.

\begin{assumption}\label{assumption:ordering}
The observation locations $\{x_i\}_{i=1}^M$ are sorted according to the maximin ordering: It holds that
\[x_1 \in \underset{ 1 \le i \le M}{\arg \max } \operatorname{dist}\left(x_i, \partial \D\right),  \]
and, for $2 \le i \le M,$
\[ x_{i} \in \underset{x \in \{x_{i}, \ldots, x_M\}}{\arg \max } \operatorname{dist}\bigl(x, \{x_1, \ldots, x_{i-1}\} \cup \partial \D \bigr).  \]
\end{assumption}


    


The following result will be proved in Section \ref{sec:proof_Cholesky}.
\begin{theorem}\label{thm:mainresult2} 
Suppose that Assumptions \ref{assumption:operator}, \ref{assumption:data}, and \ref{assumption:ordering} hold. Let $\Omega=UU^{\top}$ be the upper-triangular Cholesky factorization of $\Omega.$  For any $r>0$, there are three constants $C_1, C_2, C_3>0$ depending only on $r,\delta,d,s, \|\L\|, \|\L^{-1}\|$ such that if
\[
N\ge C_1 \log^2(1/\mathsf{h}) \log^d (N/\mathsf{h}),
\]
then there exists an estimator $\widehat{U}=\widehat{U}(Z_1,Z_2,\ldots,Z_N)$ which satisfies with probability at least $1-C_2\exp(-r(\log N)^d)$ that
\[
\frac{\|\widehat{U}-U\|}{\|U\|}\le C_3   \log(1/\mathsf{h})\bigg(\frac{\log^d (N/\mathsf{h})}{N}\bigg)^{1/2}.
\]
\end{theorem}

Similar to Theorem \ref{thm:mainresult1}, Theorem \ref{thm:mainresult2} shows that only $N\gtrsim \log^{d+2}(1/\mathsf{h})\asymp (\log(M))^{d+2}$ samples are needed to accurately estimate the upper-triangular Cholesky factor of the precision matrix $\Omega\in \R^{M\times M}$. This represents a significant improvement compared to the general sample size requirement $N\gtrsim M\asymp \mathsf{h}^{-d}$. To the best of our knowledge, Theorem \ref{thm:mainresult2} is the first result in the literature to address Cholesky factor estimation for Gaussian processes under the general Assumption \ref{assumption:operator} on the precision operator, while employing the natural maximin ordering of discrete observation locations as specified by Assumptions \ref{assumption:data} and \ref{assumption:ordering}.

Notice that the upper bound in Theorem \ref{thm:mainresult2} contains a factor $\log (1/\mathsf{h})$ not present in Theorem \ref{thm:mainresult1}. This additional factor arises from a logarithmic term in the perturbation bound for Cholesky factorization in Lemma \ref{lemma:pertubation_Cholesky}, where the logarithmic dependence on the matrix size is intrinsic and cannot be eliminated in general \cite[Section 3]{edelman1995parlett}. However, we shall see in Remark \ref{rem:removinglogfactor} that the additional logarithmic factor in Theorem \ref{thm:mainresult2} can be avoided when estimating instead a modified block Cholesky factor.


\subsection{Precision operator estimation on the lattice graph}\label{ssec:lattice}
 
To streamline our analysis of precision estimation in the setting of Subsection \ref{ssec:precision_operator}, here we consider the simplified problem of precision operator estimation on a lattice graph. Our set-up and main result retain the essential features of Subsection \ref{ssec:precision_operator}: For sparse precision operators on a lattice, Theorem \ref{thm:mainresult3} below establishes an estimation bound with an explicit dependence on the condition number.

Let $\mcG_{d,p}=[p]\times[p]\times\cdots\times [p]$ be a $d$-dimensional lattice where $[p]=\{1,2,\ldots,p\}$, and denote the number of variables by $M=p^d.$ (Points in the lattice can be interpreted as observation locations in the set-up of the two previous subsections.) Let $Z=(Z(t)$ : $\left.t \in \mathcal{G}_{d,p}\right)$ be a centered Gaussian process on the lattice graph $\mathcal{G}_{d,p}$ with covariance operator $\Sigma= (\sigma( t,t^{\prime}))_{ t,t^{\prime} \in \mathcal{G}_{d,p}}$, where $\sigma(t,t^{\prime})=\mathrm{Cov}(Z(t),Z(t^{\prime}))$. Note that the covariance operator $\Sigma$ is defined over the Cartesian product space $\mathcal{G}_{d,p} \times \mathcal{G}_{d,p}$, that is, $\Sigma \in \mathbb{R}^{\mathcal{G}_{d,p} \times \mathcal{G}_{d,p}}$. We denote the precision operator, the inverse of $\Sigma$, by $\Omega=(\omega(t,t^{\prime}))_{t,t^{\prime} \in \mathcal{G}_{d,p}}$. Given data $\{Z_n\}_{n=1}^N$ consisting of $N$ independent copies of $Z$, we are interested in estimating the precision operator $\Omega\in \mathbb{R}^{\mathcal{G}_{d,p} \times \mathcal{G}_{d,p}}$ under $\ell_2(\mathcal{G}_{d,p})\to \ell_2(\mathcal{G}_{d,p})$ operator norm. 
We will study estimation in the following precision operator class: 
\begin{align*}
\mathcal{F}_{d,p} =\left\{\Omega: \Omega \succ 0, \sum_{t^{\prime}: \|t^{\prime}-t\|_{1} \geq k}|\omega( t,t^{\prime})| \leq \|\Omega\| e^{-k}, \ \forall \ k>0 \text { and } t \in \mathcal{G}_{d,p} \right\},
\end{align*}
where $\|t-t^{\prime}\|_{1}=\sum_{i=1}^d |t_i-t^{\prime}_i|$. As will become clear in Section \ref{sec:proof_smooth}, our motivation to consider the class $\mathcal{F}_{d,p}$ is that
the precision matrix introduced in Equation \ref{eq:precisiondef} belongs to this class, provided that the precision operator $\L$ satisfies Assumption  \ref{assumption:operator} and the observation locations $\{x_i\}_{i=1}^M$ satisfy Assumption \ref{assumption:data}.

We next construct an estimator of the precision operator on the lattice graph. We describe the construction in three steps:  (1) partitioning the lattice into blocks; (2) defining local estimators in each block; and (3) combining the local estimators to define the global estimator. 

\paragraph{Partition}
We follow the partition scheme in \cite{cai2016minimax}. We start by dividing the lattice $\mathcal{G}_{d,p}$ into blocks of size $b\times \cdots \times b$ for some $b$ to be determined. Let $I_j=\left\{(j-1)b+1,(j-1)b+2,\ldots,jb\right\}$ for $j=1,2,\ldots, S-1$ and $I_{S}=\left\{(S-1)b+1,\ldots,p\right\}$ where $S=\lceil p/b \rceil$. For $\mathbf{j}=(j_1,j_2,\ldots,j_d)\in [S]\times [S]\times \cdots \times [S]$, we define the block indexed by $\mathbf{j}$ as
\[
B_{\mathbf{j}}=I_{j_1}\times I_{j_2}\times \cdots \times I_{j_d}.
\]
For a linear operator $A:\ell_2(\mathcal{G}_{d,p})\to \ell_2(\mathcal{G}_{d,p})$, we define its restriction to $B_{\mathbf{j}}\times B_{\mathbf{j^{\prime}}}$ as
\[
A_{\mathbf{jj^{\prime}}}:=A_{B_\mathbf{j}\times B_{\mathbf{j^{\prime}}}}=(a(t,t^{\prime}))_{t\in B_{\mathbf{j}},t^{\prime}\in B_{\mathbf{j^{\prime}}}}.
\]

\paragraph{Local estimators} For indices $\mathbf{j}$ and $\mathbf{j^{\prime}}$ satisfying $\|\mathbf{j}-\mathbf{j^{\prime}}\|_{\infty}\le 1$, we aim to construct a local estimator $T_{\mathbf{j}\mathbf{j^{\prime}}}$ for the sub-block $\Omega_{\mathbf{j}\mathbf{j^{\prime}}}$ of the precision operator.

For any $\mathbf{j}\in [S]\times [S]\times \cdots \times [S]$, we define its $\lambda$-neighborhood as 
\[N_{\mathbf{j},\lambda}:=\Big\{\mathbf{j^{\prime}}\in[S]\times [S]\times \cdots \times [S]  :\|\mathbf{j^{\prime}}-\mathbf{j}\|_{\infty}\le \lambda\Big\}, \quad W_{\mathbf{j},\lambda}:=\Bigg\{t\in \mathcal{G}_{d,p}: t\in \bigcup_{\mathbf{j^{\prime}}\in N_{\mathbf{j},\lambda}} B_{\mathbf{j^{\prime}}} \Bigg\}.
\] 
We consider the local covariance operator on $W_{\mathbf{j},\lambda}$
\[
\Sigma_{N_{\mathbf{j},\lambda}}:=(\sigma(t,t^{\prime}))_{t\in W_{\mathbf{j},\lambda},t^{\prime}\in W_{\mathbf{j},\lambda}},
\]
and its sample counterpart
\[
\widehat{\Sigma}_{N_{\mathbf{j},\lambda}}:=(\widehat{\sigma}(t,t^{\prime}))_{t\in W_{\mathbf{j},\lambda},t^{\prime}\in W_{\mathbf{j},\lambda}},\quad \widehat{\sigma}(t,t^{\prime}):=\frac{1}{N}\sum_{n=1}^N Z_n(t) Z_n(t^{\prime}).
\]

For $\mathbf{j}, \mathbf{j^{\prime}}$ satisfying $\|\mathbf{j}-\mathbf{j^{\prime}}\|_{\infty}\le 1$, we define the local estimator of $\Omega_{\mathbf{jj^{\prime}}}$ as
\[
T_{\mathbf{jj^{\prime}}}:=  \big(\big(\widehat{\Sigma}_{N_{\mathbf{j},2}}\big)^{-1}\big)_{\mathbf{jj^{\prime}}}.
\]
Note that here we take $\lambda=2$ and recall that $\big(\big(\widehat{\Sigma}_{N_{\mathbf{j},2}}\big)^{-1}\big)_{\mathbf{jj^{\prime}}}$ is the restriction of $\big(\widehat{\Sigma}_{N_{\mathbf{j},2}}\big)^{-1}$ to $B_{\mathbf{j}}\times B_{\mathbf{j^{\prime}}}$. As part of our analysis, we will show that, for suitable choice of block-size $b$, $\widehat{\Sigma}_{N_{\mathbf{j},2}}$ is indeed invertible with high probability.  

\paragraph{Global estimator}
Based on the local estimators for all $\mathbf{j},\mathbf{j}^{\prime}$ satisfying $\|\mathbf{j}-\mathbf{j}^{\prime}\|_{\infty}\le 1$ , we define $\widetilde{\Omega}$ as
\[
\big(\widetilde{\Omega}\big)_{\mathbf{j}\mathbf{j^{\prime}}}:=\begin{cases}
  T_{\mathbf{j}\mathbf{j^{\prime}}}  & \text{if} \ \|\mathbf{j}-\mathbf{j^{\prime}}\|_{\infty}\le 1,  \\
  0   & \text{otherwise}.
\end{cases}
\]
Finally, we define the estimator for $\Omega$ as the symmetrized version of $\widetilde{\Omega}$
\begin{align}\label{eq:estimator_lattice}
\widehat{\Omega}:=\frac{1}{2}\big(\widetilde{\Omega}+\widetilde{\Omega}^{\top}\big).
\end{align}


We are now ready to state our main result concerning precision operator estimation on the lattice graph, which provides a high-probability, spectral-norm bound on the estimation error with an explicit dependence on the condition number $\kappa(\Omega)=\|\Omega\|\|\Omega^{-1}\|$ of the precision. The proof can be found in Section \ref{sec:proof_lattice}.

\begin{theorem}\label{thm:mainresult3}  For any $\Omega\in \mathcal{F}_{d,p}$ and any constant $r>0$, there are two constants $C_1, C_2>0$ depending only on $r,d$ such that if
\[
N\ge C_1\left((\log N)^d+\log p+(\log \kappa(\Omega))^d\right), \quad  p>\log (N\kappa(\Omega)),
\]
then the estimator $\widehat{\Omega}$ in \eqref{eq:estimator_lattice} with $b=\lceil \log (N\kappa(\Omega))\rceil$, satisfies with probability at least $1-2\left(\frac{\log (N\kappa(\Omega))}{p}\right)^{rd}e^{-(r+1)(\log N)^d} $ that
\[
\frac{\|\widehat{\Omega}-\Omega\|}{\|\Omega\|}\le C_2\left(\sqrt{\frac{(\log N)^d}{N}}+\sqrt{\frac{\log p}{N}}+\sqrt{\frac{(\log \kappa(\Omega))^d}{N}}\ \right).
\]
\end{theorem}

\begin{remark}\label{remark:main3}
If $p\le \log (N\kappa(\Omega))$, we define $\widehat{\Omega}:=(\widehat{\Sigma})^{-1}$. By the argument used in display \eqref{eq:local_aux2} in the proof of Proposition \ref{prop:local_estimate}, there exists a constant $C>1$ such that if $N>4C^2 p^d$, then for all $t\ge 1$, it holds with probability at least $1-e^{-N/4C^2}-e^{-t}$ that
\begin{align*}
\frac{\|\widehat{\Omega}-\Omega\|}{\|\Omega\|}=\frac{\|(\widehat{\Sigma})^{-1}-\Omega\|}{\|\Omega\|}\le 2C\Biggl(\sqrt{\frac{p^d}{N}}
\vee \sqrt{\frac{t}{N}} \Biggr).
\end{align*}
For any $r>0$, taking $t=(r+1)(\log N)^d$ yields that
\begin{align*}
\mathbb{P} \left[ \frac{\|\widehat{\Omega}-\Omega\|}{\|\Omega\|}\ge C_2\left(\sqrt{\frac{(\log N)^d}{N}}+\sqrt{\frac{(\log \kappa(\Omega))^d}{N}}\ \right) \right]
&\le 2e^{-(r+1)(\log N)^d}\\
&\le 2\left(\frac{\log (N\kappa(\Omega))}{p}\right)^{rd}e^{-(r+1)(\log N)^d},
\end{align*}
provided that $N\ge C_1\left((\log N)^d+(\log \kappa(\Omega))^d\right)$, where $C_1, C_2>0$ are two constants depending only on $r$ and $d$. Therefore, when $p\le \log(N\kappa(\Omega))$, the inverse of the sample covariance achieves the same high-probability bound as stated in Theorem \ref{thm:mainresult3}.
\end{remark}

\begin{remark}\label{remark:GGM}
In this remark, we offer a regression-based interpretation of the local estimator $T_{\mathbf{jj^{\prime}}}$ and heuristically explain why it is expected to perform well.

For a vector $Z$ and index $i$, we use the notation $Z_{\sim i}=((Z_j) : j\ne i)$. Let $Z$ be drawn from the Gaussian distribution with zero mean and precision matrix $\Omega\in \R^{p\times p}$.  Then, for any $i$, $Z_i \mid Z_{\sim i}=z_{\sim i}$ is distributed as $\mathcal{N}\left(\left\langle w^{(i)}, z_{\sim i}\right\rangle, 1 / \Omega_{i i}\right)$ where $w^{(i)}$ is the vector with $w_j^{(i)}=-\Omega_{i j} / \Omega_{i i}$, i.e.
    \begin{align}\label{eq:precision_regression}
Z_i = -\sum_{j\ne i}\frac{\Omega_{ij}}{\Omega_{ii}}Z_j+\varepsilon,\qquad \varepsilon\sim \mathcal{N} \Bigl(0,\frac{1}{\Omega_{ii}} \Bigr).
    \end{align}
This establishes a basic connection between precision matrix estimation and linear regression. Suppose we observe $\boldsymbol{Z}\in \R^{N\times p}$, where each row is independently drawn from $\mathcal{N}(0,\Omega^{-1})$. We use the notation $\boldsymbol{Z_i}$ and $\boldsymbol{Z_{\sim i}}$ to denote the $i-$th column of $\boldsymbol{Z}$ and the matrix obtained by removing the $i-$th column from $\boldsymbol{Z}$, respectively.  Consider
    \[
    \boldsymbol{Z_i} = \boldsymbol{Z_{\sim i}} \boldsymbol{\beta}+\boldsymbol{\varepsilon}, \quad \boldsymbol{Z_i}\in \R^{N\times 1}, \ \boldsymbol{Z_{\sim i}}\in \R^{N\times (p-1)}, \ \boldsymbol{\beta}\in \R^{(p-1)\times 1}, \ \boldsymbol{\varepsilon}\in \R^{N\times 1}. 
    \]
If we perform ordinary least squares (OLS) regression, the regression coefficient and residue read as
    \[
    \boldsymbol{\widehat{\beta}}=(\boldsymbol{Z_{\sim i}}^{\top}\boldsymbol{Z_{\sim i}})^{-1}\boldsymbol{Z_{\sim i}}^{\top} \boldsymbol{Z_i},\quad  \widehat{ \|\boldsymbol{\varepsilon}\|^2} = \|\boldsymbol{Z_i}-\boldsymbol{Z_{\sim i}}\boldsymbol{\widehat{\beta}}\|^2=\boldsymbol{Z_i}^{\top}\boldsymbol{Z_i}-\boldsymbol{Z_i}^{\top} \boldsymbol{Z_{\sim i}}(\boldsymbol{Z_{\sim i}}^{\top}\boldsymbol{Z_{\sim i}})^{-1}\boldsymbol{Z_{\sim i}}^{\top} \boldsymbol{Z_i}.
    \]
     Given $\boldsymbol{\widehat{\beta}}$ and $\widehat{ \|\boldsymbol{\varepsilon}\|^2}$, from the relation \eqref{eq:precision_regression}, the natural plug-in estimators for $\Omega_{ii}$ and $\Omega_{i,\sim i}$ are $N/\widehat{ \|\boldsymbol{\varepsilon}\|^2}$ and $-N\boldsymbol{\widehat{\beta}}^{\top}/\widehat{ \|\boldsymbol{\varepsilon}\|^2}$, respectively. 
    
On the other hand, by Lemma \ref{lemma:block_matrix_inverse}, the entries in the $i$-th row of the inverse of sample covariance matrix are
    \[
    \left(\frac{\boldsymbol{Z}^{\top }\boldsymbol{Z}}{N}\right)^{-1}_{i,i}=N(\boldsymbol{Z_i}^{\top}\boldsymbol{Z_i}-\boldsymbol{Z_i}^{\top} \boldsymbol{Z_{\sim i}} (\boldsymbol{Z_{\sim i}}^{\top} \boldsymbol{Z_{\sim i}})^{-1}\boldsymbol{Z_{\sim i}}^{\top} \boldsymbol{Z_{i}})^{-1}=N/\widehat{ \|\boldsymbol{\varepsilon}\|^2},
    \]
    and
    \begin{align*}
    \left(\frac{\boldsymbol{Z}^{\top }\boldsymbol{Z}}{N}\right)^{-1}_{i,\sim i}&=-N\left(\boldsymbol{Z_i}^{\top}\boldsymbol{Z_i}-\boldsymbol{Z_i}^{\top} \boldsymbol{Z_{\sim i}} (\boldsymbol{Z_{\sim i}}^{\top} \boldsymbol{Z_{\sim i}})^{-1}\boldsymbol{Z_{\sim i}}^{\top} \boldsymbol{Z_{i}}\right)^{-1}\boldsymbol{Z_i}^{\top}\boldsymbol{Z_{\sim i}} (\boldsymbol{Z_{\sim i}}^{\top}\boldsymbol{Z_{\sim i}})^{-1}\\
    &=-N\boldsymbol{\widehat{\beta}}^{\top}/\widehat{ \|\boldsymbol{\varepsilon}\|^2}.
    \end{align*}
     In other words, for estimating the $i$-th row of the precision matrix, $\Omega_{i,:}$, the $i$-th row of the inverse sample covariance provides the same estimate as the OLS+plug-in method. Thanks to this relation, we give a regression-based interpretation for the local estimator $T_{\mathbf{jj^{\prime}}}=  \big(\big(\widehat{\Sigma}_{N_{\mathbf{j},2}}\big)^{-1}\big)_{\mathbf{jj^{\prime}}}$ as follows. For each $t\in B_{\mathbf{j}}$, we perform a local OLS regression of $Z(t)$ against the variables in its neighborhood, $(Z(t^{\prime}):t^{\prime}\in W_{\mathbf{j},2}\backslash \{t\} )$, obtaining the regression coefficient and residue. Then, we construct the natural plug-in estimator for the slice $((\Sigma_{N_{\mathbf{j},2}})^{-1}(t,t^{\prime}):t^{\prime}\in W_{\mathbf{j},2})$ based on \eqref{eq:precision_regression}.  Equivalently, this gives $\big(\big(\widehat{\Sigma}_{N_{\mathbf{j},2}}\big)^{-1}(t,t^{\prime}):t^{\prime}\in W_{\mathbf{j},2}\big)$. Finally, $T_{\mathbf{jj^{\prime}}}$ is defined as the sub-matrix of $\big(\widehat{\Sigma}_{N_{\mathbf{j},2}}\big)^{-1}$ corresponding to indices in $B_{\mathbf{j}}\times B_{\mathbf{j^{\prime}}}$, serving as our estimate for $\Omega_{\mathbf{jj^{\prime}}}$. The following two-step approximation is expected to hold
     \[
T_{\mathbf{jj^{\prime}}}=\big(\big(\widehat{\Sigma}_{N_{\mathbf{j},2}}\big)^{-1}\big)_{\mathbf{jj^{\prime}}}\overset{\text{(i)}}{\approx} \big((\Sigma_{N_{\mathbf{j},2}})^{-1}\big)_{\mathbf{jj^{\prime}}} \overset{\text{(ii)}}{\approx}\Omega_{\mathbf{jj^{\prime}}}.
     \]
     The first approximation (i) follows from the relatively small size of $\Sigma_{N_{\mathbf{j},2}}$, which allows the population covariance to be well approximated by its sample counterpart. The second approximation (ii) is expected to hold because, when conditioned on the variables in $Z(t)$'s neighborhood, $Z(t)$ and the remaining variables are approximately uncorrelated due to the exponential decay property of $\Omega \in \mathcal{F}_{d,p}$. As a result, the regression of $Z(t)$ on its neighborhood or on all other variables would yield similar results. Since $\Omega$ can be interpreted through OLS regression, as previously discussed, this supports the validity of (ii). The rigorous proof is provided in Section \ref{sec:proof_lattice}.

The idea of using nearest neighbor conditioning sets and assuming conditional independence outside these sets is known as the Vecchia approximation \cite{vecchia1988estimation} in spatial statistics. More specifically, the Vecchia approximation expresses the joint density as a product of conditional densities with \emph{reduced} conditioning sets, thereby  enhancing computational efficiency. This method has been widely applied in Gaussian process regression \cite{gramacy2015local,gramacy2016speeding}. Furthermore, the works \cite{schäfer2021compression,schäfer2021sparse} implicitly form a Vecchia approximation by virtue of sparsity. For recent developments on this topic, we refer the reader to \cite{katzfuss2021general,kang2024asymptotic}. Our local regression approach for estimating precision operator on the lattice graph can be viewed as an application of Vecchia's idea, specifically tailored to statistical estimation problems. Beyond its computational benefits, our results in this paper demonstrate that Vecchia's idea also plays a crucial role in constructing efficient statistical estimators with favorable sample complexity.

 
     
\end{remark}

\section{Precision estimation for Gaussian processes}\label{sec:proof_smooth}



This section studies precision estimation for Gaussian processes in the general setting of Assumptions \ref{assumption:operator} and \ref{assumption:data}. We first present two lemmas that characterize the essential properties of the precision matrix being estimated. These properties were established through a series of works by H. Owhadi and C. Scovel in the context of operator-adapted wavelets, fast solvers, and numerical homogenization \cite{owhadi2017multigrid,owhadi2017universal,owhadi2019operator}. Leveraging these results, we proceed to prove our first main result, Theorem \ref{thm:mainresult1}.


\begin{lemma}[Lemma 15.43 in \cite{owhadi2019operator}, Proposition E.1 in \cite{chen2024sparse}]\label{lemma:magnitude}
Suppose that Assumptions \ref{assumption:operator} and \ref{assumption:data} hold. There exists a constant $C$ depending only on $\|\mathcal{L}\|, \left\|\mathcal{L}^{-1}\right\|, d, s$ and $\delta$ such that
\begin{align*}
& C^{-1} \mathsf{h}^{d-2 s} \leq \lambda_{\max }(\Omega) \leq C \mathsf{h}^{d-2 s}, \\
& C^{-1}\mathsf{h}^d \leq \lambda_{\min }(\Omega), \\
& C^{-1} \mathsf{h}^{d-2 s} \leq \Omega_{ii} \leq C \mathsf{h}^{d-2 s}, \quad 1\le i \le M.
\end{align*}
\end{lemma}
\begin{remark}
In our notation, $\Omega=\mathsf{h}^d A$ where $A$ is the stiffness matrix in \cite[Lemma 15.43]{owhadi2019operator}, see \cite[Example 4.5]{owhadi2019operator}.
\end{remark}

\begin{lemma}[Screening effect, Theorem 9.6 in \cite{owhadi2019operator}]\label{lemma:sparsity} Suppose that Assumptions \ref{assumption:operator} and \ref{assumption:data} hold. Let $u \sim \mathcal{N}(0,\L^{-1})$. For $i\ne j,$ the conditional correlation between $u(x_i)$ and $u(x_j)$ given $u(x_{\ell})$ for $\ell\ne i,j$ satisfies
    \[
    |\mathrm{Cor}(u(x_i),u(x_j)|u(x_\ell),\ell\ne i,j)|=\frac{|\Omega_{ij}|}{\sqrt{\Omega_{ii}\Omega_{jj}}}\le Ce^{-C^{-1}\|x_i-x_j\|/\mathsf{h}},
    \]
    where the constant $C$ depends only on $\|\L\|,\|\L^{-1}\|,d,s$ and $\delta$.
\end{lemma}

We are now ready to prove Theorem \ref{thm:mainresult1}. The core idea of the proof is to relate the homogeneously scattered observation locations $\{x_i\}_{i=1}^M$, as described in Assumption \ref{assumption:data}, to a regular lattice in the domain $\D=[0,1]^d$ by using Hall's marriage theorem \cite{hall1987representatives}. This formalizes the heuristic that ``homogeneously scattered points'' can be treated as approximately equivalent to ``points on a regular lattice''. Specifically, we apply Hall's marriage theorem to establish \textbf{Claim I} below in the proof of Theorem 2.3, which shows that each scattered observation location can be matched to a nearby point on a slightly larger regular lattice with minimal geometric distortion. 

This geometric matching preserves key structural properties, i.e., the exponential decay of correlations, ensuring that the original problem retains its essential statistical features under this transformation. As a result, we are able to reduce the original estimation problem to that of estimating a slightly enlarged precision operator $\bar{\Omega}$ on the lattice. We then verify that the estimation of $\bar{\Omega}$ falls into the framework of precision operator estimation on the lattice graph, as developed in Subsection \ref{ssec:lattice} and Section \ref{sec:proof_lattice}. Finally, we apply Theorem \ref{thm:mainresult3} and Remark \ref{remark:main3} to complete the proof.


We first recall the graph-theoretic formulation of Hall's marriage theorem \cite{hall1987representatives}. Let $\mathsf{G}=(\mathsf{X},\mathsf{Y},\mathsf{E})$ be a finite bipartite graph with bipartite sets $\mathsf{X}$ and $\mathsf{Y}$ and edge set $\mathsf{E}$. For any subset $\mathsf{W} \subset \mathsf{X}$, let $\mathsf{N}_{\mathsf{G}}(\mathsf{W})$ denote its neighborhood in $\mathsf{G}$, i.e., the set of all vertices in $\mathsf{Y}$ that are adjacent to at least one vertex in $\mathsf{W}$. An $\mathsf{X}$-perfect matching is a matching (a set of disjoint edges) which covers every vertex in $\mathsf{X}$. Hall's marriage theorem provides a necessary and sufficient condition for the existence of an $\mathsf{X}$-perfect matching in the bipartite graph $\mathsf{G}$.
\begin{lemma}[Hall's marriage theorem \cite{hall1987representatives}]\label{lemma:Hall}
There is an $\mathsf{X}$-perfect matching in $\mathsf{G}=(\mathsf{X},\mathsf{Y},\mathsf{E})$ if and only if for every subset $\mathsf{W}\subset \mathsf{X}$, $\mathrm{Card}(\mathsf{W})\le \mathrm{Card}(\mathsf{N}_{\mathsf{G}}(\mathsf{W}))$.    
\end{lemma}

\begin{proof}[Proof of Theorem \ref{thm:mainresult1}]

Let $p$ be a positive integer to be determined later. Following the notation in Subsection \ref{ssec:lattice}, let $\mathcal{G}_{d,p}=[p]\times[p]\times\cdots\times [p]$ be a $d$-dimensional lattice where $[p]=\{1,2,\ldots,p\}$. For every $t=(t_1,t_2,\ldots,t_d)\in\mathcal{G}_{d,p} $, we define
\[
y_t:=\frac{t}{p+1}=\left(\frac{t_1}{p+1},\frac{t_2}{p+1},\ldots,\frac{t_d}{p+1}\right)\in \D = [0,1]^d.
\]

We claim that the following holds. 

\textbf{Claim I:} There is a constant $c_1>0$ depending only on $d$ and $\delta$ such that if $p\ge \frac{1}{c_1\mathsf{h}}$, then there exists $A\subset \mathcal{G}_{d,p}$ with $\mathrm{Card}(A)=M$, and a bijection $\pi: A \to \{1,2,\ldots,M\}$ such that
\begin{align}\label{eq:mainresult1_aux1}
\max_{t\in A} \|x_{\pi(t)}-y_{t}\|\le \mathsf{h}.
\end{align}

We prove \textbf{Claim I} by applying Hall's marriage theorem (Lemma \ref{lemma:Hall}). \nc We construct a bipartite graph $\mathsf{G}=(\mathsf{X},\mathsf{Y},\mathsf{E})$ with bipartite sets $\mathsf{X}=\{1,2,\ldots,M\}$ and $\mathsf{Y}=\mathcal{G}_{d,p}$. The edge set $\mathsf{E}$ of the graph is defined as
\[
(i,t)\in \mathsf{E} \iff \|x_i-y_t\|\le \mathsf{h},\quad i\in \mathsf{X},\, t\in \mathsf{Y}.
\]

Observe that in our set-up, if $\mathsf{h}>\frac{1}{p+1}$, then for every subset $\mathsf{W}\subset \mathsf{X}$, it holds that
\begin{align}\label{eq:mainthm1_aux1}
\mathrm{Card}(\mathsf{N}_{\mathsf{G}}(\mathsf{W}))\gtrsim_{d} \mathrm{Card}(\mathsf{W}) \bigg(\frac{ \mathsf{h}}{\frac{1}{p+1}}\bigg)^d / \bigg(\frac{\mathsf{h}}{\delta \mathsf{h}}\bigg)^d=\mathrm{Card}(\mathsf{W})\big(\delta\mathsf{h}(p+1)\big)^d,
\end{align}
where in the first inequality we used two facts: (I) a volume argument shows that the degree of $i\in \mathsf{X}$ is roughly $\big(\frac{ \mathsf{h}}{\frac{1}{p+1}}\big)^d$, provided $\mathsf{h} >\frac{1}{p+1}$; (II) by Assumption \ref{assumption:data} (3), the degree of $t\in \mathsf{Y}$ is bounded above by $\big(\frac{\mathsf{h}}{\delta \mathsf{h}}\big)^d$ up to a constant depending on $d$. From \eqref{eq:mainthm1_aux1}, if we take $p= \lceil \frac{1}{c_1 \mathsf{h}}\rceil$ for a sufficiently small $c_1\in (0,1)$ which depends only on $\delta$ and $d$, then $\mathsf{h}>\frac{1}{p+1}$ and $ \mathrm{Card}(\mathsf{N}_{\mathsf{G}}(\mathsf{W}))\ge \mathrm{Card}(\mathsf{W})$ for every subset $\mathsf{W}\subset \mathsf{X}$. By Hall's marriage theorem, there is a $\mathsf{X}$-perfect matching. From our construction, every vertex $i\in \mathsf{X}$ is only connected to $t\in \mathsf{Y}$ satisfying $\|y_t-x_i\|\le \mathsf{h}$. The $\mathsf{X}$-perfect matching gives the desired subset $A\subset \mathsf{Y}=\mathcal{G}_{d,p}$ and bijection $\pi$ between $A$ and $\mathsf{X}$, and so the \textbf{Claim I} follows.

We denote by $A^{c}$ the complement of $A$ in $\mathcal{G}_{d,p}$. Based on the set $A\subset \mathcal{G}_{d,p}$ and the bijection $\pi$, we define a precision operator $\bar{\Omega}=(\bar{\omega}(t,t^{\prime}))_{t,t^{\prime} \in \mathcal{G}(d,p)}$ on the lattice $\mathcal{G}_{d,p}$ by
\begin{align*}
\bar{\omega}(t,t^{\prime}):=\begin{cases}\Omega_{\pi(t)\pi(t^{\prime})}\quad & t,t^{\prime} \in A,\\
\mathbf{1}_{\{t=t^{\prime}\}} & t,t^{\prime}\in A^{c},\\
0 & t \in A, t^{\prime}\in A^c \text{ or } t \in A^c, t^{\prime}\in A,
\end{cases}
\end{align*}
where $\mathbf{1}_{S}$ denotes the indicator function of the set $S$.

We denote by $\bar{\Omega}|_{A}$ the restriction of $\bar{\Omega}$ to the Cartesian product space $A\times A$. Observe that $\bar{\Omega}|_{A}$ can be interpreted as a permuted version of $\Omega$, while $\bar{\Omega}$ can be viewed as $\bar{\Omega}|_{A}$ extended by padding with an identity operator on $A^c$. Given $A$ and $\pi$, estimating $\Omega$ is equivalent to estimating $\bar{\Omega}|_{A}$. Next, we show that estimation of $\bar{\Omega}$ aligns with the framework of precision operator estimation on the lattice graph, as established in Subsection \ref{ssec:lattice} and Section \ref{sec:proof_lattice}. We then apply Theorem \ref{thm:mainresult3} and Remark \ref{remark:main3} to complete the proof of Theorem \ref{thm:mainresult1}.

We first verify the exponential decay property of the precision operator $\bar{\Omega}$ on the lattice. For any $t\in A^c$, it holds that $\sum_{t^{\prime}: \|t^{\prime}-t\|_1 \ge k}|\bar{\omega}(t, t^{\prime})|=0$ for $k\ge 1$. For any $t\in A$ and $k\ge 8\sqrt{d}/c_1$, it holds that
\begin{align*}
  \sum_{t^{\prime}: \|t^{\prime}-t\|_1 \ge k}|\bar{\omega}(t, t^{\prime})|&\overset{\text{(i)}}{=}\sum_{t^{\prime}\in A: \|y_{t^{\prime}}-y_{t}\|_1 \ge \frac{k}{p+1}}|\Omega_{\pi(t)\pi(t^{\prime})}|\\
  & \overset{\text{(ii)}}{\le}\sum_{t^{\prime}\in A: \|y_{t^{\prime}}-y_{t}\| \ge \frac{k}{\sqrt{d}(p+1)}}|\Omega_{\pi(t)\pi(t^{\prime})}|\\
  &\overset{\text{(iii)}}{\le}\sum_{t^{\prime}\in A:\|x_{\pi(t^{\prime})}-x_{\pi(t)}\|\ge \frac{c_1k\mathsf{h}}{4\sqrt{d}}} |\Omega_{\pi(t)\pi(t^{\prime})}|,
\end{align*}
where (i) follows from the definition $y_t=t/(p+1),y_{t^{\prime}}=t^{\prime}/(p+1)$, $\bar{\omega}(t,t^{\prime})=\Omega_{\pi(t)\pi(t^{\prime})}$ for $t,t^{\prime}\in A$, and $\bar{\omega}(t,t^{\prime})=0$ for $t\in A,t^{\prime}\in A^c$; (ii) follows by $\|y_{t^{\prime}}-y_{t}\|_1\le \sqrt{d}\|y_{t^{\prime}}-y_{t}\|$; (iii) follows by applying triangle inequality and \eqref{eq:mainresult1_aux1}:
\begin{align*}
\|x_{\pi(t^{\prime})}-x_{\pi(t)}\|&\ge\|y_{t^{\prime}}-y_{t}\|-\|x_{\pi(t)}-y_{t}\|-\|x_{\pi(t^{\prime})}-y_{t^{\prime}}\|\\
&\ge \frac{k}{\sqrt{d}(p+1)} -2\mathsf{h}=\left(\frac{k}{\sqrt{d}}\cdot \frac{c_1}{1+c_1\mathsf{h}}-2\right)\mathsf{h}\ge\left(\frac{c_1k}{2\sqrt{d}}-2\right)\mathsf{h} \ge \frac{c_1k}{4\sqrt{d}}\mathsf{h}.
\end{align*}

We then apply Lemma \ref{lemma:sparsity} and use the fact $\Omega_{\pi(t)\pi(t)}\le \|\Omega\|$ for all $t\in A$ to obtain that
\begin{align*}
  \sum_{t^{\prime}\in A:\|x_{\pi(t^{\prime})}-x_{\pi(t)}\|\ge \frac{c_1k\mathsf{h}}{4\sqrt{d}}} |\Omega_{\pi(t)\pi(t^{\prime})}|&=\sum_{t^{\prime}\in A:\|x_{\pi(t^{\prime})}-x_{\pi(t)}\|\ge \frac{c_1 k\mathsf{h}}{4\sqrt{d}}} \frac{|\Omega_{\pi(t)\pi(t^{\prime})}|}{\sqrt{\Omega_{\pi(t)\pi(t)}\Omega_{\pi(t^{\prime})\pi(t^{\prime})}}}\cdot \sqrt{\Omega_{\pi(t)\pi(t)}\Omega_{\pi(t^{\prime})\pi(t^{\prime})}}  \\
  &\le \|\Omega\| \sum_{t^{\prime}\in A:\|x_{\pi(t^{\prime})}-x_{\pi(t)}\|\ge \frac{c_1 k\mathsf{h}}{4\sqrt{d}}} Ce^{-C^{-1}\|x_{\pi(t^{\prime})}-x_{\pi(t)}\|/\mathsf{h}}\\
    &\le C\|\Omega\| \sum_{\ell=k}^{\infty}\sum_{t^{\prime}\in A:\|x_{\pi(t^{\prime})}-x_{\pi(t)}\|\in \left(\frac{c_1\ell\mathsf{h}}{4\sqrt{d}},\frac{c_1(\ell+1)\mathsf{h}}{4\sqrt{d}}\right] } e^{-C^{-1}\|x_{\pi(t^{\prime})}-x_{\pi(t)}\|/\mathsf{h}}\\
    &\le C\|\Omega\| \sum_{\ell=k}^{\infty} e^{-C^{-1}\frac{c_1\ell}{4\sqrt{d}}} \sum_{t^{\prime}\in A:\|x_{\pi(t^{\prime})}-x_{\pi(t)}\|\in \left(\frac{c_1\ell\mathsf{h}}{4\sqrt{d}},\frac{c_1(\ell+1)\mathsf{h}}{4\sqrt{d}}\right] } 1\\
    &\lesssim \|\Omega\|\sum_{\ell=k}^{\infty} e^{-C^{-1}\frac{c_1\ell}{4\sqrt{d}}}\ell^{d-1}\\
    &\lesssim \|\Omega\| e^{-ck},
\end{align*}
where $c>0$ is a constant depending only on $\|\L\|,\|\L^{-1}\|,s,d$ and $\delta$. Therefore, we have so far shown that for any $t\in A$ and $k\ge 8\sqrt{d}/c_1$,
\begin{align}\label{eq:mainresult1_aux2}
 \sum_{t^{\prime}: \|t^{\prime}-t\|_1 \ge k}|\bar{\omega}(t, t^{\prime})|\lesssim \|\Omega\|e^{-c k}=\|\bar{\Omega}|_{A}\|e^{-ck}\le \|\bar{\Omega}\|e^{-ck}.
\end{align}

For the condition number of $\bar{\Omega}$, we have from Lemma \ref{lemma:magnitude} 
\begin{align}\label{eq:mainresult1_aux3}
\kappa(\bar{\Omega})=\|\bar{\Omega}\|\|\bar{\Omega}^{-1}\|= \big(\|\bar{\Omega}|_{A}\|\lor 1\big)\big(\|(\bar{\Omega}|_{A})^{-1}\|\lor 1\big)= \big(\|\Omega\|\lor 1\big)\big(\|\Omega^{-1}\|\lor 1\big)\lesssim \mathsf{h}^{-2s}.
\end{align}

To estimate $\bar{\Omega}$, we additionally draw $N$ samples $\{W_n\}_{n=1}^N \iid \Nc(0,I_{\mathrm{Card}(A^c)})$ and concatenate them with the original data $\{Z_n\}_{n=1}^N$ to form new samples $\{\bar{Z}_n\}_{n=1}^N$, where $\bar{Z}_n=(Z_n,W_n)$ for $1\le n\le N$. Consequently, $\{\bar{Z}_n\}_{n=1}^N\iid \Nc(0,\bar{\Omega}^{-1})$. Recall that $p= \lceil \frac{1}{c_1 \mathsf{h}} \rceil\asymp 1/\mathsf{h}$. Combining the exponential decay property of the precision operator \eqref{eq:mainresult1_aux2} and the polynomial growth of its condition number \eqref{eq:mainresult1_aux3}, a straightforward application of Theorem \ref{thm:mainresult3} and Remark \ref{remark:main3} implies that there are three constants $C_1, C_2, C_3>0$ depending only on $r,\delta,d,s, \|\L\|, \|\L^{-1}\|$ such that if
\[
N\ge C_1 \log^d (N/\mathsf{h}),
\]
then there exists an estimator $\widehat{\bar{\Omega}}=\widehat{\bar{\Omega}}(\bar{Z}_1,\bar{Z}_2,\ldots,\bar{Z}_N)$ such that with probability at least $1-C_2(\mathsf{h}\log (N/\mathsf{h}))^{dr}\exp(-(r+1)(\log N)^d)$
\[
\frac{\|\widehat{\bar{\Omega}}-\bar{\Omega}\|}{\|\bar{\Omega}\|}\le C_3 \bigg(\frac{\log^d (N/\mathsf{h})}{N}\bigg)^{1/2},
\]
which implies
\[
\|(\widehat{\bar{\Omega}})|_{A}-\bar{\Omega}|_{A}\|\le \|\widehat{\bar{\Omega}}-\bar{\Omega}\|\le C_3\|\bar{\Omega}\|\bigg(\frac{\log^d (N/\mathsf{h})}{N}\bigg)^{1/2}\asymp \|\bar{\Omega}|_{A}\|\bigg(\frac{\log^d (N/\mathsf{h})}{N}\bigg)^{1/2},
\]
where in the last step we used that $\|\bar{\Omega}\|= \|\bar{\Omega}|_{A}\|\lor 1\asymp \|\bar{\Omega}|_{A}\|=\|\Omega\|$.
Since $\bar{\Omega}|_{A}$ is a permuted version of $\Omega$, and so estimating $\bar{\Omega}|_{A}$ is equivalent to estimating $\Omega$, the proof is complete.
\end{proof}

\begin{remark}
When applying Theorem \ref{thm:mainresult3} in the context of Theorem \ref{thm:mainresult1}, we set the block-size parameter $b= \lceil \log (N\kappa(\Omega))\rceil \asymp \log(N/\mathsf{h}) $.
    For practical implementation, $b$ can be set as $b=\lceil c_0 \log(N/\mathsf{h}) \rceil $, where $c_0$ may be taken as a fixed constant or chosen empirically through cross-validation \cite{bickel2008covariance,cai2011adaptive}.
\end{remark}

\begin{remark}
Besides its statistical advantage, our estimation procedure is also efficient from a computational perspective, because it is defined by combining \emph{local} estimators. In particular, for the estimator constructed in Section 2.3 ---which estimates the precision operator on the lattice graph--- each local estimator $T_{\mathbf{jj^{\prime}}}$ involves computing the matrix $\widehat{\Sigma}_{N_{\mathbf{j},2}}$ and taking its inverse. A naive implementation of this step incurs a computational cost of 
\[
\mathcal{O}(N(b^{d})^2+(b^d)^3)=\mathcal{O}(N \log^{2d}(N/\mathsf{h})+ \log^{3d}(N/\mathsf{h}))=\mathcal{O}(N \log^{2d}(N/\mathsf{h}))=\mathcal{O}(N\log^{2d}(MN))
\]
under the setting of Theorem \ref{thm:mainresult1}. Here, we have used the assumption $N\ge C_1\log^d(N/\mathsf{h})$ and the relation $M\asymp\mathsf{h}^{-d}$. Aggregated over all $M$ such local estimators, the total computational cost is approximately $\mathcal{O}(M N\log^{2d}
(MN))$. For comparison, constructing a full sample covariance matrix of size $M\times M$ from $N$ samples incurs a cost of $\mathcal{O}(M^2N)$, and inverting the matrix requires an additional $\mathcal{O}(M^{3})$, leading to a total cost of $\mathcal{O}(M^2N+M^3)$. While our $\mathcal{O}(M N\log^{2d}
(MN))$ may not be the optimal computational complexity, we believe it remains practical in many scenarios, particularly due to its localized nature and potential for parallel implementation.
\end{remark}

\begin{remark}
Hall's marriage theorem provides a rigorous and elegant way to crystallize the intuition that ``homogeneously scattered points'' are approximately equivalent to ``points on a regular lattice''. This matching idea may also be useful to analyze other problems involving homogeneously scattered points. In many cases, the analysis is more tractable when the points are arranged on a regular lattice. Our matching technique provides a geometric bridge between the scattered and lattice settings, potentially enabling a reduction of the scattered case to the lattice case. This strategy is employed in our proof of Theorem 2.3 for estimating precision matrices. We conjecture that this approach may help simplify certain arguments in the literature ---for example, in establishing exponential decay of correlations--- and we regard this as an interesting direction for future research.
\end{remark}

\section{Cholesky factor estimation}\label{sec:proof_Cholesky}

In the general setting of Assumptions \ref{assumption:operator} and \ref{assumption:data}, this section focuses on estimating the Cholesky factor of $\Omega$ under the maximin ordering of observation locations formalized in Assumption \ref{assumption:ordering}. We begin by reviewing the background on the hierarchical structure induced by the maximin ordering, as described in \cite{schäfer2021compression}. Next, we present a block-Cholesky decomposition formula in Lemma \ref{lemma:decomposition} and a bounded condition number property for the diagonal block matrices in Lemma \ref{lemma:bounded_cond}, both of which were established in \cite{schäfer2021compression}. With these results, we construct estimators for each block in the Cholesky factorization. By invoking Theorem \ref{thm:mainresult1} and combining information across all scales in the hierarchical structure, we establish our second main result, Theorem \ref{thm:mainresult2}.

As illustrated in \cite[Section 5.2, Figure 3.5]{schäfer2021compression}, the maximin ordering has a hidden hierarchical structure: the successive points in the maximin ordering can be grouped into levels, so that the Cholesky factorization in the maximin ordering falls in the following setting:

\begin{example}[Example 5.1 in \cite{schäfer2021compression}]\label{assumption:hierarchy}
    Let $q \in \mathbb{N}$. For $h,\delta \in(0,1)$, let $\left\{x_i\right\}_{i \in I^{(1)}} \subset\left\{x_i\right\}_{i \in I^{(2)}} \subset \cdots \subset$ $\left\{x_i\right\}_{i \in I^{(q)}}$ be a nested hierarchy of points  in $\D$ that are homogeneously distributed at each scale in the sense that, for $k \in \{1, \ldots, q\}:$  
    
(1) $\sup _{x \in \D} \min _{i \in I^{(k)}}\|x-x_i\| \leq h^k$;

(2) $\min _{i \in I^{(k)}} \inf _{x \in \partial \D}\|x-x_i\| \geq \delta h^k$; 

(3) $\min _{i, j \in I^{(k)}: i \neq j}\|x_i-x_j\| \geq \delta h^k$.
\end{example}

Throughout this section, we set the parameter $h=1/2$. Let $J^{(1)}:=I^{(1)}$ and $J^{(k)}:=I^{(k)} \backslash I^{(k-1)}$ for $k \in\{2, \ldots, q\}$, then we get a partition of $I = \{1, \ldots, M\}$, i.e. $I=I^{(q)}=\bigcup_{1 \leq k \leq q} J^{(k)}$. Note that $I^{(k)}=\bigcup_{1 \leq k^{\prime} \leq k} J^{\left(k^{\prime}\right)}$. We represent $M\times M$ matrices as $q \times q$ block matrices according to this partition. Given an $M \times M$ matrix $W$, we write $W_{k, l}$ for the $(k, l)^{\text {th}}$ block of $W$ and $W_{k_1: k_2, l_1: l_2}$ for the sub-matrix of $W$ defined by blocks ranging from $k_1$ to $k_2$ and $l_1$ to $l_2$. We interpret the $\left\{J^{(k)}\right\}_{1 \leq k \leq q}$ as labelling a hierarchy of scales with $J^{(1)}$ representing the coarsest and $J^{(q)}$ the finest. The finest scale $\mathsf{h} = h^q = 2^{-q}$ is determined by the observational scale, so that $ q \asymp \log (1/ \mathsf{h}).$ 
We emphasize that the hierarchical structure in Example \ref{assumption:hierarchy} follows directly from Assumptions \ref{assumption:data} and \ref{assumption:ordering}, along with the relation $q\asymp\log(1/\mathsf{h})$, rather than from introducing an additional assumption.

Under the nested hierarchy presented in Example \ref{assumption:hierarchy}, we define
\[
\Sigma^{(k)}:=\Sigma_{1:k,1:k},\quad \Omega^{(k)}:=\Sigma^{(k),-1} \quad \text { for } 1 \leq k \leq q ,
\]
where $\Sigma^{(k),-1}$ denotes the inverse of $\Sigma^{(k)}.$
For any $\lambda \in I$, we set $|\lambda|:=k$ if and only if $\lambda\in J^{(k)}$. To relate our setting to \cite{schäfer2021compression}, we define
\begin{align}\label{eq:transform}
  \Theta:= \widetilde{D}^{-1} \Sigma\widetilde{D}^{-1}, \quad \widetilde{D} := \diag \left(h^{-d |\lambda | /2}: \lambda\in I\right),
\end{align}
and
\begin{align}\label{eq:transform2}
\Theta^{(k)}:=\widetilde{D}^{(k),-1}\Sigma^{(k)}\widetilde{D}^{(k),-1} ,\quad \widetilde{D}^{(k)} := \diag \left(h^{-d |\lambda | /2}: \lambda\in I^{(k)}\right),\notag \\
A^{(k)}:=\Theta^{(k),-1}, \quad B^{(k)}:=A_{k, k}^{(k)} \quad \text { for } 1 \leq k \leq q .
\end{align}

The following block-Cholesky decomposition of $\Theta$ plays a crucial role in establishing the exponential decay of Cholesky factors and proving rigorous complexity bounds for sparse Cholesky factorization in \cite{schäfer2021compression}. It also serves as the foundation for our statistical estimation of the upper-triangular Cholesky factor of $\Omega$.

\begin{lemma}[Lemma 5.3 in \cite{schäfer2021compression}]\label{lemma:decomposition}
We have $\Theta=\bar{L} D \bar{L}^{\top}$, with $\bar{L}$ and $D$ defined by
\begin{align*}
D:=\left(\begin{array}{cccc}
B^{(1),-1} & 0 & \ldots & 0 \\
0 & B^{(2),-1} & \ddots & \vdots \\
\vdots & \ddots & \ddots & \vdots \\
0 & 0 & \ldots & B^{(q),-1}
\end{array}\right), \bar{L}:=\left(\begin{array}{cccc}
\mathrm{Id} & \cdots & \cdots & 0 \\
B^{(2),-1} A_{2,1}^{(2)} & \ddots & 0 & \vdots \\
\vdots & \ddots & \ddots & \vdots \\
B^{(q),-1} A_{q, 1}^{(q)} & \ldots & B^{(q),-1} A_{q, q-1}^{(q)} & \mathrm{Id}
\end{array}\right)^{-1}.
\end{align*}
In particular, if $\widetilde{L}$ is the lower-triangular Cholesky factor of $D$, then the lower-triangular Cholesky factor $L$ of $\Theta$ is given by $L=\bar{L} \widetilde{L}$.
\end{lemma}



Notice that for $k\ge 2$, $B^{(k),-1}$ can be interpreted as conditional covariance \cite[(5.21)]{schäfer2021compression} 
\begin{align*}
B^{(k),-1}&=((\Theta^{(k),-1})_{k,k})^{-1}=\Theta^{(k)}_{k,k}-\Theta^{(k)}_{k,1:k-1}(\Theta^{(k)}_{1:k-1,1:k-1})^{-1}\Theta^{(k)}_{1:k-1,k}\\
&=h^{dk}\Big(\Sigma^{(k)}_{k,k}-\Sigma^{(k)}_{k,1:k-1}(\Sigma^{(k)}_{1:k-1,1:k-1})^{-1}\Sigma^{(k)}_{1:k-1,k}\Big)\\
&=h^{dk}\mathrm{Cov}(u(x_i),i\in J^{(k)} \mid u(x_i),i\in I^{(k-1)}).
\end{align*}
One of the key properties of block Cholesky factorization in Lemma \ref{lemma:decomposition} is that $B^{(k),-1}$ is well-conditioned, as stated in the following lemma. We recall that $h=1/2$ in this section. 
\begin{lemma}[Theorem 5.8 in \cite{schäfer2021compression}, Theorem 5.18 in \cite{owhadi2019operator}]\label{lemma:bounded_cond}
There exists a constant $C$ depending only on $\|\mathcal{L}\|,\left\|\mathcal{L}^{-1}\right\|, d, s$, and $\delta$ such that
\begin{enumerate}
    \item $C^{-1} h^{-2 (k-1)s} \Id \leq B^{(k)} \leq C h^{-2 k s} \Id $.
    \item $\kappa(B^{(k)})=\mathrm{cond}(B^{(k)})\le C h^{-2s}$.
\end{enumerate}
\end{lemma}


Before proving Theorem \ref{thm:mainresult2}, the next proposition constructs estimators for $B^{(k)}$ and its lower-triangular Cholesky factor, and establishes high-probability bounds on the estimation error. We defer the proof of Proposition \ref{prop:B_estimates} to the end of this section.

\begin{proposition}\label{prop:B_estimates}
  For $1\le k\le q$ and any $r>0$, there are three constants $C_1, C_2, C_3>0$ depending only on $r,\delta,d,s, \|\L\|, \|\L^{-1}\|$ such that if
\[
N\ge C_1 k^2\log^d (N/h^k),
\]
then there exists an estimator $\widehat{B^{(k)}}$ such that with probability at least \[
1-C_2(h^k\log (N/h^k))^{dr}\exp(-(r+1)(\log N)^d),
\]
$\widehat{B^{(k)}}$ is positive definite and
\[
    \frac{\|\widehat{B^{(k)}}-B^{(k)}\|}{\|B^{(k)}\|}\le C_3  \bigg(\frac{\log^d (N/h^k)}{N}\bigg)^{1/2},\quad \frac{\|(\widehat{B^{(k)}})^{-1}-B^{(k),-1}\|}{\|B^{(k),-1}\|}\le C_3\bigg(\frac{\log^d (N/h^k)}{N}\bigg)^{1/2},
    \]
    \[
     \frac{\|\widehat{\widetilde{L}^{(k)}}-\widetilde{L}^{(k)}\|}{\|\widetilde{L}^{(k)}\|}\le  C_3\bigg(\frac{k^2 \log^d (N/h^k)}{N}\bigg)^{1/2},\quad \frac{\|(\widehat{\widetilde{L}^{(k)}})^{-1}-\widetilde{L}^{(k),-1}\|}{\|\widetilde{L}^{(k),-1}\|}\le C_3 \bigg(\frac{k^2\log^d (N/h^k)}{N}\bigg)^{1/2},
    \]
  where $\widetilde{L}^{(k)},\widehat{\widetilde{L}^{(k)}}$ are the lower-triangular Cholesky factors of $B^{(k),-1},(\widehat{B^{(k)}})^{-1}$ respectively, i.e. $B^{(k),-1}=\widetilde{L}^{(k)} (\widetilde{L}^{(k)})^{\top}$ and $(\widehat{B^{(k)}})^{-1}=\widehat{\widetilde{L}^{(k)}} (\widehat{\widetilde{L}^{(k)}})^{\top}$.
\end{proposition}

Now we are ready to prove Theorem \ref{thm:mainresult2}.

\begin{proof}[Proof of Theorem \ref{thm:mainresult2}] 

From Lemma \ref{lemma:decomposition} and our definition \eqref{eq:transform},
\begin{align*}
    \Omega=\Sigma^{-1}=\widetilde{D}^{-1} \Theta^{-1} \widetilde{D}^{-1}=\widetilde{D}^{-1} \bar{L}^{-\top} D^{-1} \bar{L}^{-1} \widetilde{D}^{-1}=\widetilde{D}^{-1} \bar{L}^{-\top}  \widetilde{L}^{-\top}\widetilde{L}^{-1} \bar{L}^{-1} \widetilde{D}^{-1}=:UU^{\top},
\end{align*}
where $U=\left(\widetilde{L}^{-1} \bar{L}^{-1} \widetilde{D}^{-1}\right)^{\top}$ is the upper-triangular Cholesky factor of $\Omega$. Further, let
$\widetilde{L}^{(k)}$ be the lower-triangular Cholesky factor of $B^{(k),-1}$, i.e. $B^{(k),-1}=\widetilde{L}^{(k)} (\widetilde{L}^{(k)})^{\top}$. Then $\widetilde{L}=\diag(\widetilde{L}^{(k)}: 1\le k\le q)$ is the lower-triangular Cholesky factor of $D$. From Lemma \ref{lemma:decomposition}, $U^{\top}$ admits the form
    \begin{align*}
    U^{\top}&=\widetilde{L}^{-1} \bar{L}^{-1} \widetilde{D}^{-1}=\left(\begin{array}{cccc}
h^{\frac{d}{2}}\widetilde{L}^{(1),-1} & \cdots & \cdots & 0 \\
h^{\frac{d}{2}}(\widetilde{L}^{(2)})^{\top} A_{2,1}^{(2)} & h^{\frac{2d}{2}}\widetilde{L}^{(2),-1} & 0 & \vdots \\
\vdots & \ddots & \ddots & \vdots \\
h^{\frac{d}{2}}(\widetilde{L}^{(q)})^{\top} A_{q, 1}^{(q)} & \ldots & h^{\frac{(q-1)d}{2}}(\widetilde{L}^{(q)})^{\top} A_{q, q-1}^{(q)} & h^{\frac{qd}{2}}\widetilde{L}^{(q),-1}
\end{array}\right)\\
&=\left(\begin{array}{cccc}
h^{\frac{d}{2}}\widetilde{L}^{(1),-1} & \cdots & \cdots & 0 \\
h^{-\frac{2d}{2}}(\widetilde{L}^{(2)})^{\top} \Omega_{2,1}^{(2)} & h^{\frac{2d}{2}}\widetilde{L}^{(2),-1} & 0 & \vdots \\
\vdots & \ddots & \ddots & \vdots \\
h^{-\frac{qd}{2}}(\widetilde{L}^{(q)})^{\top} \Omega_{q, 1}^{(q)} & \ldots & h^{-\frac{qd}{2}}(\widetilde{L}^{(q)})^{\top} \Omega_{q, q-1}^{(q)} & h^{\frac{qd}{2}}\widetilde{L}^{(q),-1}
\end{array}\right)
    \end{align*}
where in the last step we used $A^{(k)}=\widetilde{D}^{(k)}\Omega^{(k)}\widetilde{D}^{(k)}$ for $2\le k\le q$ by \eqref{eq:transform2}, so the $(k, l)^{\text {th}}$ ($2\le k\le q, 1\le l\le k-1$) block of $U^{\top}$ is equal to
\[
(U^{\top})_{k,l}=h^{\frac{ld}{2}} (\widetilde{L}^{(k)})^{\top}A^{(k)}_{k,l}=h^{\frac{ld}{2}} (\widetilde{L}^{(k)})^{\top}h^{-\frac{d(k+l)}{2}}\Omega^{(k)}_{k,l}=h^{-\frac{kd}{2}} (\widetilde{L}^{(k)})^{\top}\Omega^{(k)}_{k,l},
\]
which can be expressed compactly as
\[
(U^{\top})_{k,1:k-1}=h^{-\frac{kd}{2}} (\widetilde{L}^{(k)})^{\top}\Omega^{(k)}_{k,1:k-1},\quad 2\le k\le q.
\]
 Here $\Omega^{(k)}$ is the precision matrix with observation resolution $\mathsf{h}=h^{k}$. By Theorem \ref{thm:mainresult1}, Remark \ref{remark:main3}, and Proposition \ref{prop:B_estimates}, for $1\le k\le q$ and any $r>0$, there are three constants $\widetilde{C}_1, \widetilde{C}_2, \widetilde{C}_3>0$ depending only on $r,\delta,d,s, \|\L\|, \|\L^{-1}\|$ such that if
\[
N\ge \widetilde{C}_1 k^2\log^d (N/h^k),
\]
then there exist estimators $\widehat{\Omega^{(k)}}, \widehat{\widetilde{L}^{(k)}}$ such that it holds with probability at least 
\[1-\widetilde{C}_2(h^k\log (N/h^k))^{dr}\exp(-(r+1)(\log N)^d)
\]
that
 \begin{align}
 \label{eq:cholesky_aux1}
    \|\widehat{\Omega^{(k)}}-\Omega^{(k)}\|&\le \widetilde{C}_3 \|\Omega^{(k)}\|\bigg(\frac{\log^d (N/h^k)}{N}\bigg)^{1/2}\le  \widetilde{C}_4 h^{k(d-2s)}\bigg(\frac{\log^d (N/h^k)}{N}\bigg)^{1/2},  \\
\label{eq:cholesky_aux2}
\|\widehat{\widetilde{L}^{(k)}}-\widetilde{L}^{(k)}\| &\le  \widetilde{C}_3\|\widetilde{L}^{(k)}\|\bigg(\frac{k^2 \log^d (N/h^k)}{N}\bigg)^{1/2}\le \widetilde{C}_4 h^{ks} \bigg(\frac{k^2\log^d (N/h^k)}{N}\bigg)^{1/2}, \\
\label{eq:cholesky_aux3}
\|(\widehat{\widetilde{L}^{(k)}})^{-1}-\widetilde{L}^{(k),-1}\| &\le \widetilde{C}_3 \|\widetilde{L}^{(k),-1}\|\bigg(\frac{k^2\log^d (N/h^k)}{N}\bigg)^{1/2}\le \widetilde{C}_4 h^{-ks}\bigg(\frac{k^2\log^d (N/h^k)}{N}\bigg)^{1/2},
    \end{align}
    where we used $\|\Omega^{(k)}\|\asymp h^{k(d-2s)}$ by Lemma \ref{lemma:magnitude}, $\|\widetilde{L}^{(k)}\|=\sqrt{\|B^{(k),-1}\|}\asymp h^{ks}$ and $\|\widetilde{L}^{(k),-1}\|=\sqrt{\|B^{(k)}\|}\asymp h^{-ks}$ by Lemma \ref{lemma:bounded_cond}.

 Now we define our estimator $\widehat{U^{\top}}$ for $U^{\top}$ in a blockwise manner
 \begin{align*}
     (\widehat{U^{\top}})_{k,1:k-1}&:= h^{-\frac{kd}{2}} (\widehat{\widetilde{L}^{(k)}})^{\top} (\widehat{\Omega^{(k)}})_{k,1:k-1},\quad 2\le k\le q, \\
     (\widehat{U^{\top}})_{k,k} &:= h^{\frac{kd}{2}} (\widehat{\widetilde{L}^{(k)}})^{-1},\quad 1\le k\le q.
 \end{align*}
By triangle inequality, we have that
\begin{align}\label{eq:cholesky_aux4}
&\|(\widehat{U^{\top}})_{k,1:k-1}-(U^{\top})_{k,1:k-1}\| \le h^{-\frac{kd}{2}}\|(\widehat{\widetilde{L}^{(k)}})^{\top} (\widehat{\Omega^{(k)}})_{k,1:k-1}-(\widetilde{L}^{(k)})^{\top}\Omega^{(k)}_{k,1:k-1}\|\nonumber\\
&\le h^{-\frac{kd}{2}}\left(\|(\widetilde{L}^{(k)})^{\top}\|\| \big(\widehat{\Omega^{(k)}}\big)_{k,1:k-1}-\Omega^{(k)}_{k,1:k-1}\|+\|(\widehat{\widetilde{L}^{(k)}})^{\top}-(\widetilde{L}^{(k)})^{\top}\|\|(\widehat{\Omega^{(k)}})_{k,1:k-1}\|\right) \nonumber \\
& \le h^{-\frac{kd}{2}}\left(\|\widetilde{L}^{(k)}\|\| \widehat{\Omega^{(k)}}-\Omega^{(k)}\|+\|\widehat{\widetilde{L}^{(k)}}-\widetilde{L}^{(k)}\|\big(\|\widehat{\Omega^{(k)}}-\Omega^{(k)}\|+\|\Omega^{(k)}\|\big)\right) \nonumber \\
&\overset{\eqref{eq:cholesky_aux1}+\eqref{eq:cholesky_aux2}}{\lesssim}h^{-\frac{kd}{2}}\cdot h^{ks}\cdot h^{k(d-2s)}\bigg(\frac{k^2\log^d (N/h^k)}{N}\bigg)^{1/2}\left(2+\bigg(\frac{k^2\log^d (N/h^k)}{N}\bigg)^{1/2}\right)\nonumber
\\
&\le \widetilde{C}_5 h^{k(d/2-s)}\bigg(\frac{k^2\log^d (N/h^k)}{N}\bigg)^{1/2}.
\end{align}

For the block matrix on the diagonal ($1\le k\le q$), we have
\begin{align}\label{eq:cholesky_aux5}
\|(\widehat{U^{\top}})_{k,k}-(U^{\top})_{k,k}\|=h^{\frac{kd}{2}}\|(\widehat{\widetilde{L}^{(k)}})^{-1}-\widetilde{L}^{(k),-1}\|\overset{\eqref{eq:cholesky_aux3}}{\le}\widetilde{C}_4 h^{k(d/2-s)}\bigg(\frac{k^2\log^d (N/h^k)}{N}\bigg)^{1/2}.
\end{align}

We are now ready to bound $\|\widehat{U^{\top}}-U^{\top}\|$. We will repeatedly use the following spectral norm inequalities for block matrices
\begin{equation}\label{eq:cholesky_aux6}
\|\begin{bmatrix} X & Y \end{bmatrix}\|\le \|X\|+\|Y\|,\quad \| \begin{bmatrix} X  \\ Y \end{bmatrix} \|\le \|X\|+\|Y\|.
\end{equation}

Since $\|U^{\top}\|=\sqrt{\|\Omega\|}\asymp h^{q(d/2-s)}$, we have
\begin{align*}
\frac{\|\widehat{U^{\top}}-U^{\top}\|}{\|U^{\top}\|} &\asymp h^{q(s-d/2)}\|\widehat{U^{\top}}-U^{\top}\|\\
&\overset{\eqref{eq:cholesky_aux6}}{\le} h^{q(s-d/2)}\sum_{k=1}^{q} \|(\widehat{U^{\top}})_{k,1:k}-(U^{\top})_{k,1:k}\|\\
&\overset{\eqref{eq:cholesky_aux6}}{\le} h^{q(s-d/2)}\sum_{k=1}^q\left(\|(\widehat{U^{\top}})_{k,1:k-1}-(U^{\top})_{k,1:k-1}\|+ \|(\widehat{U^{\top}})_{k,k}-(U^{\top})_{k,k}\| \right)\\
&\overset{\eqref{eq:cholesky_aux4}+\eqref{eq:cholesky_aux5}}{\le} h^{q(s-d/2)}\sum_{k=1}^q (\widetilde{C}_4+\widetilde{C}_5) h^{k(d/2-s)}\bigg(\frac{k^2\log^d (N/h^k)}{N}\bigg)^{1/2}\\
&\le (\widetilde{C}_4+\widetilde{C}_5)\bigg(\frac{q^2\log^d (N/h^q)}{N}\bigg)^{1/2}\bigg(\sum_{k=1}^{q}h^{(q-k)(s-d/2)}\bigg)\\
&\lesssim  \bigg(\frac{q^2\log^d (N/h^q)}{N}\bigg)^{1/2}.
\end{align*}

The probability that all of the above inequalities hold (at all scales $k \in [1, q]$) is at least
\begin{align*}
&1-\sum_{k=1}^q \widetilde{C}_2(h^k\log (N/h^k))^{dr}\exp(-(r+1)(\log N)^d)\\
&= 1-\widetilde{C}_2 \exp(-(r+1)(\log N)^d)\sum_{k=1}^{q} h^{kdr}(\log N+k\log 2)^{dr} \\
&\ge 1-\widetilde{C}_2 \exp(-(r+1)(\log N)^d)\sum_{k=1}^{q} h^{kdr}((2\log N)^{dr}+(2k\log 2)^{dr})\\
&\ge 1-C_2 \exp(-r(\log N)^d),
\end{align*}
where in the last line we used that $\sum_{k=1}^q h^{kdr}\le C(d,r)$, $\sum_{k=1}^q h^{kdr}k^{dr}\le C(d,r)$, and $\exp(-(\log N)^d)(\log N)^{dr}\le C(d,r)$, where $C(d,r)$ is a constant depending only on $d$ and $r$. Recall that $\mathsf{h}=h^q=2^{-q}$, i.e. $ q\asymp \log(1/\mathsf{h})$, and so the proof is complete.
\end{proof}

\begin{proof}[Proof of Proposition \ref{prop:B_estimates}]
    First, for $1\le k\le q$, we have
    \[
    B^{(k)}=A^{(k)}_{k,k}=(\Theta^{(k),-1})_{k,k}=\left(\widetilde{D}^{(k)} \Omega^{(k)} \widetilde{D}^{(k)}\right)_{k,k}=h^{-kd} \Omega^{(k)}_{k,k}.
    \]

    
    Since $\Omega^{(k)}$ is the precision matrix at observational scale $\mathsf{h}=h^{k}$, Theorem \ref{thm:mainresult1} provides an estimator $\widehat{\Omega^{(k)}}$ such that if $N\ge \widetilde{C}_1 \log^d (N/h^k)$, then it holds with probability at least $1-\widetilde{C}_2(h^k\log (N/h^k))^{dr}\exp(-(r+1)(\log N)^d)$ that
    \[
    \frac{\|\widehat{\Omega^{(k)}}-\Omega^{(k)}\|}{\|\Omega^{(k)}\|}\le \widetilde{C}_3\bigg(\frac{\log^d (N/h^k)}{N}\bigg)^{1/2}.
    \]

    We define
\[
\widehat{B^{(k)}}:=h^{-kd}\big(\widehat{\Omega^{(k)}}\big)_{k,k}.
\]
Then,
\begin{align*}
\frac{\|\widehat{B^{(k)}}-B^{(k)}\|}{\|B^{(k)}\|}&=\frac{\|\big(\widehat{\Omega^{(k)}}\big)_{k,k}-\Omega^{(k)}_{k,k}\|}{\|\Omega^{(k)}_{k,k}\|}\le \frac{\|\widehat{\Omega^{(k)}}-\Omega^{(k)}\|}{\|\Omega^{(k)}_{k,k}\|}\\
&\le \widetilde{C}_3\frac{\|\Omega^{(k)}\|}{\|\Omega^{(k)}_{k,k}\|}\bigg(\frac{\log^d (N/h^k)}{N}\bigg)^{1/2}\le \widetilde{C}_4\bigg(\frac{\log^d (N/h^k)}{N}\bigg)^{1/2}=:\varepsilon_{B^{(k)}},
\end{align*}
where the last step follows by Lemma \ref{lemma:magnitude}, the diagonal entries of $\Omega^{(k)}$ are of order $(h^{k})^{d-2s}$, so $\|\Omega^{(k)}_{k,k}\|\gtrsim (h^{k})^{d-2s}\asymp \|\Omega^{(k)}\|$. 

Suppose $N\ge 4\kappa^2(B^{(k)})(\widetilde{C}_1\lor \widetilde{C}_4^2)\log^d(N/h^k)$. Then, $\kappa(B^{(k)})\varepsilon_{B^{(k)}}\le 1/2$ and
\[
\lambda_{\min}(\widehat{B^{(k)}})\ge \lambda_{\min}(B^{(k)})-\|\widehat{B^{(k)}}-B^{(k)}\|\ge \lambda_{\min}(B^{(k)})(1-\kappa(B^{(k)})\varepsilon_{B^{(k)}})\ge \frac{1}{2}\lambda_{\min}(B^{(k)})>0,
\]
and hence $\widehat{B^{(k)}}$ is positive definite. By Lemma \ref{lemma:perturbation_inverse},
\[
\frac{\|(\widehat{B^{(k)}})^{-1}-B^{(k),-1}\|}{\|B^{(k),-1}\|}\le \frac{\kappa(B^{(k)})\varepsilon_{B^{(k)}}}{1-\kappa(B^{(k)})\varepsilon_{B^{(k)}}} \le 2\kappa(B^{(k)})\widetilde{C}_4\bigg(\frac{\log^d (N/h^k)}{N}\bigg)^{1/2}.
\]

Since $B^{(k),-1}$ and $(\widehat{B^{(k)}})^{-1}$ are positive definite, we consider their Cholesky factorization $B^{(k),-1}=\widetilde{L}^{(k)} (\widetilde{L}^{(k)})^{\top}$ and $(\widehat{B^{(k)}})^{-1}=\widehat{\widetilde{L}^{(k)}} (\widehat{\widetilde{L}^{(k)}})^{\top}$. By Lemma \ref{lemma:pertubation_Cholesky}, 
\begin{align*}
\frac{\|\widehat{\widetilde{L}^{(k)}}-\widetilde{L}^{(k)}\|}{\|\widetilde{L}^{(k)}\|} &\le \bigl(2\log_2(\mathrm{size}(B^{(k)})) +4 \bigr) \kappa(B^{(k),-1})\frac{\|(\widehat{B^{(k)}})^{-1}-B^{(k),-1}\|}{\|B^{(k),-1}\|}\\
&\le 4 \bigl(\log_2(\mathrm{size}(B^{(k)}))+2 \bigr)\kappa^2(B^{(k)})\widetilde{C}_4\bigg(\frac{\log^d (N/h^k)}{N}\bigg)^{1/2}\\
&\le \widetilde{C}_5\bigg(\frac{k^2 \log^d (N/h^k)}{N}\bigg)^{1/2}=:\varepsilon_{\widetilde{L}^{(k)}},
\end{align*}
where in the last step we used $h=1/2$, $\mathrm{size}(B^{(k)}))\asymp h^{-kd}$, $\kappa(B^{(k)})\lesssim h^{-2s}\asymp O(1)$ by Lemma \ref{lemma:bounded_cond}. Again, suppose $N>4\kappa^2(\widetilde{L}^{(k)})(\widetilde{C}_1\lor \widetilde{C}_5^2) k^2 \log^d(N/h^k)$.  Then, $\kappa(\widetilde{L}^{(k)})\varepsilon_{\widetilde{L}^{(k)}}\le 1/2$ and by Lemma \ref{lemma:perturbation_inverse},
\[
\frac{\|(\widehat{\widetilde{L}^{(k)}})^{-1}-\widetilde{L}^{(k),-1}\|}{\|\widetilde{L}^{(k),-1}\|}\le \frac{\kappa(\widetilde{L}^{(k)})\varepsilon_{\widetilde{L}^{(k)}}}{1-\kappa(\widetilde{L}^{(k)})\varepsilon_{\widetilde{L}^{(k)}}}\le 2\kappa(\widetilde{L}^{(k)})\widetilde{C}_5\bigg(\frac{k^2 \log^d (N/h^k)}{N}\bigg)^{1/2}.
\]
To conclude the proof, notice that the  constants
\[
C_1:=\left(4\kappa^2(B^{(k)})(\widetilde{C}_1\lor \widetilde{C}_4^2)\right)\lor \left(4\kappa^2(\widetilde{L}^{(k)})(\widetilde{C}_1\lor \widetilde{C}_5^2)\right)
,\quad C_2:=\widetilde{C}_2,
\]
\[ C_3:=\left(2\kappa(B^{(k)})\widetilde{C}_4\right)\lor\left(2\kappa(\widetilde{L}^{(k)})\widetilde{C}_5\right)
\]
depend only on $r,\delta,d,s, \|\L\|, \|\L^{-1}\|,$ as desired.
\end{proof}

We now state two matrix perturbation bounds used in the proof of Proposition \ref{prop:B_estimates}. The proofs can be found in Appendix \ref{app:A}.

\begin{lemma}[Perturbation bound for matrix inverse]\label{lemma:perturbation_inverse}
 Suppose $B$ and $\widehat{B}$ are invertible and $\|\widehat{B}-B\|\le \|B\|\varepsilon_{B}$ with $\kappa(B)\varepsilon_B<1.$ Then,
 \[
 \frac{\|(\widehat{B})^{-1}-B^{-1}\|}{\|B^{-1}\|}\le \frac{\kappa(B)\varepsilon_B}{1-\kappa(B)\varepsilon_B}.
 \]
\end{lemma}

\begin{lemma}[Perturbation bound for Cholesky factorization \cite{drmavc1994perturbation,edelman1995parlett}]\label{lemma:pertubation_Cholesky}
Let $B,\widehat{B}\in \R^{M\times M}$ be positive definite matrices, and let $B=LL^{\top}, \widehat{B}=\widehat{L}\widehat{L}^{\top}$ be their Cholesky factorization. Suppose that $\|\widehat{B}-B\|\le \|B\|\varepsilon_{B}.$ Then,
\[
\frac{\|\widehat{L}-L\|}{\|L\|}< (2 \log_2 M+4)\kappa(B)\varepsilon_B.
\]
\end{lemma}

\begin{remark}\label{rem:removinglogfactor}
In this remark, we discuss how the factor $\log (1/\mathsf{h})$ which appears in our bound for Cholesky factor estimation in Theorem \ref{thm:mainresult2} can be removed when estimating instead a modified block Cholesky factor. 
    Recall from Lemma \ref{lemma:decomposition} that $\Theta=L^{*}(L^{*})^{\top},$ where $L^{*}=\bar{L}D^{1/2}$ is a block lower-triangular matrix. Since $\Omega$ is a rescaled version of $\Theta^{-1}$, it follows that $\Omega=U^{*}(U^{*})^{\top}$, where $U^{*}$ is a block upper-triangular matrix. We claim that the estimation rate for $U^{*}$ does not include the extra logarithmic factor. The key observation is that the perturbation bound for matrix square root does not involve any logarithmic factor, contrasting with the perturbation bound for Cholesky factorization in Lemma \ref{lemma:pertubation_Cholesky}. Applying a Lipschitz-type estimate for the matrix square root \cite[Lemma 2.2]{schmitt1992perturbation} gives that
\[
\|\sqrt{\widehat{B}}-\sqrt{B}\|\le \frac{1}{\sqrt{\lambda_{\min}(\widehat{B})}+\sqrt{\lambda_{\min}(B)}}\|\widehat{B}-B\|\le \frac{1}{\sqrt{\lambda_{\min}(B)}}\|\widehat{B}-B\|=\sqrt{\|B^{-1}\|}\|\widehat{B}-B\|,
\]
which implies
\begin{align}\label{eq:perturb_root}
\frac{\|\sqrt{\widehat{B}}-\sqrt{B}\|}{\|\sqrt{B}\|}\le \frac{\sqrt{\|B^{-1}\|}\|\widehat{B}-B\|}{\sqrt{\|B\|}}=\sqrt{\kappa(B)}\frac{\|\widehat{B}-B\|}{\|B\|}.
\end{align}
Based on this perturbation bound, a simple modification of the proof of Theorem \ref{thm:mainresult2} shows that if $\widehat{\widetilde{L}^{(k)}}$ and $(\widehat{\widetilde{L}^{(k)}})^{-1}$ are replaced by $(\widehat{B^{(k)}})^{-1/2}$ and $(\widehat{B^{(k)}})^{1/2}$, respectively, and \eqref{eq:perturb_root} is used in place of Lemma \ref{lemma:pertubation_Cholesky}, the resulting estimator $\widehat{U^{*}}$ for $U^{*}$ satisfies 
\[
\frac{\|\widehat{U^{*}}-U^{*}\|}{\|U^{*}\|}\lesssim \bigg(\frac{\log^d (N/\mathsf{h})}{N}\bigg)^{1/2}
\]
with high probability.
Thus, we have derived a bound for estimating $U^*$ which agrees with the 
bound for estimating the precision matrix $\Omega$ in Theorem \ref{thm:mainresult1} and improves upon the bound for estimating $U$ in Theorem \ref{thm:mainresult2}.
\end{remark}

\section{Precision operator estimation on the lattice graph}\label{sec:proof_lattice}

This section contains the proof of Theorem \ref{thm:mainresult3}. Before presenting the proof, the next result establishes a high-probability bound on the local estimation error $\|T_{\mathbf{j j^{\prime}}}-\Omega_{\mathbf{j j^{\prime}}}\|$.

\begin{proposition}\label{prop:local_estimate}
    Under the setting in Section \ref{ssec:lattice}, for any $\mathbf{j}, \mathbf{j^{\prime}}\in [S]\times[S]\times\cdots\times [S]$ satisfying $\|\mathbf{j}-\mathbf{j^{\prime}}\|_{\infty}\le 1$, there exists some universal constant $C>1$ such that if $N\ge 4C^2 5^db^d$, then for all $t\ge 1$, it holds with probability at least $1-e^{-N/4C^2}-e^{-t}$ that
 \[
 \frac{\|T_{\mathbf{j j^{\prime}}}-\Omega_{\mathbf{j j^{\prime}}}\|}{\|\Omega\|}\le 2C\Biggl(\sqrt{\frac{5^d b^d}{N}} 
\vee \sqrt{\frac{t}{N}}\Biggr)+\kappa(\Omega)e^{-3b}.
 \]
\end{proposition}
\begin{proof}
First, by triangle inequality
\begin{align}\label{eq:local_aux1}
 \|T_{\mathbf{jj^{\prime}}}-\Omega_{\mathbf{jj^{\prime}}}\|&=\left\|\Big(\big(\widehat{\Sigma}_{N_{\mathbf{j},2}}\big)^{-1}\Big)_{\mathbf{jj^{\prime}}}-\Omega_{\mathbf{jj^{\prime}}}\right\| \nonumber \\
 &\le \left\|\Big(\big(\widehat{\Sigma}_{N_{\mathbf{j},2}}\big)^{-1}\Big)_{\mathbf{jj^{\prime}}}-\Big(\big(\Sigma_{N_{\mathbf{j},2}}\big)^{-1}\Big)_{\mathbf{jj^{\prime}}}\right\|+\left\|\Big(\big(\Sigma_{N_{\mathbf{j},2}}\big)^{-1}\Big)_{\mathbf{jj^{\prime}}}-\Omega_{\mathbf{jj^{\prime}}}\right\|.
\end{align}
The first term in \eqref{eq:local_aux1} can be bounded by
\begin{align*}
\left\|\Big(\big(\widehat{\Sigma}_{N_{\mathbf{j},2}}\big)^{-1}\Big)_{\mathbf{jj^{\prime}}}-\Big(\big(\Sigma_{N_{\mathbf{j},2}}\big)^{-1}\Big)_{\mathbf{jj^{\prime}}}\right\| &\le \left\|\big(\widehat{\Sigma}_{N_{\mathbf{j},2}}\big)^{-1}-\big(\Sigma_{N_{\mathbf{j},2}}\big)^{-1}\right\|\\
&=\left\|\big(\Sigma_{N_{\mathbf{j},2}}\big)^{-1/2}\left(\left(\Sigma_{N_{\mathbf{j},2}}^{-1/2}\widehat{\Sigma}_{N_{\mathbf{j},2}}\Sigma_{N_{\mathbf{j},2}}^{-1/2}\right)^{-1}-I\right)\big(\Sigma_{N_{\mathbf{j},2}}\big)^{-1/2}\right\|\\
&\le \left\|\big(\Sigma_{N_{\mathbf{j},2}}\big)^{-1}\right\|\left\|\left(\Sigma_{N_{\mathbf{j},2}}^{-1/2}\widehat{\Sigma}_{N_{\mathbf{j},2}}\Sigma_{N_{\mathbf{j},2}}^{-1/2}\right)^{-1}-I\right\|\\
&\le \|\Omega\| \left\|\left(\Sigma_{N_{\mathbf{j},2}}^{-1/2}\widehat{\Sigma}_{N_{\mathbf{j},2}}\Sigma_{N_{\mathbf{j},2}}^{-1/2}\right)^{-1}-I\right\|,
\end{align*} 
where the last inequality follows from $\left\|\big(\Sigma_{N_{\mathbf{j},2}}\big)^{-1}\right\|=1/\lambda_{\min}(\Sigma_{N_{\mathbf{j},2}})\le 1/\lambda_{\min}(\Sigma)=\|\Omega\|$.

Note that $\Sigma_{N_{\mathbf{j},2}}^{-1/2}\widehat{\Sigma}_{N_{\mathbf{j},2}}\Sigma_{N_{\mathbf{j},2}}^{-1/2}$ has the same distribution as $\frac{1}{N}YY^\top$, where $Y\in \R^{\mathrm{Card}(W_{\mathbf{j},2})\times N}$ and $Y_{ij}\iid \mathcal{N}(0,1)$. By \cite[Corollary 2]{koltchinskii2017concentration}, there exists a constant $C>1$ such that, for all $t\ge 1$, it holds with probability at least $1-e^{-t}$ that
\begin{align*}
\left\|\Sigma_{N_{\mathbf{j},2}}^{-1/2}\widehat{\Sigma}_{N_{\mathbf{j},2}}\Sigma_{N_{\mathbf{j},2}}^{-1/2}-I\right\| &\le  C\Biggl(\sqrt{\frac{\mathrm{Card}(W_{\mathbf{j},2})}{N}} \vee \frac{\mathrm{Card}(W_{\mathbf{j},2})}{N} 
\vee \sqrt{\frac{t}{N}} \vee \frac{t}{N}\Biggr)\\
&\le C\Biggl(\sqrt{\frac{5^d b^d}{N}} \vee \frac{5^d b^d}{N} 
\vee \sqrt{\frac{t}{N}} \vee \frac{t}{N}\Biggr),
\end{align*}
where the last inequality follows since $\mathrm{Card}(W_{\mathbf{j},2})\le 5^d b^d $. Suppose that $N\ge 4 C^2 5^d b^d .$ Then, with probability at least $1-e^{-N/4C^2}$,
\[
\lambda_{\min}\Big(\Sigma_{N_{\mathbf{j},2}}^{-1/2}\widehat{\Sigma}_{N_{\mathbf{j},2}}\Sigma_{N_{\mathbf{j},2}}^{-1/2}\Big)\ge 1-C\Biggl(\sqrt{\frac{5^db^d}{N}} \vee \frac{5^db^d}{N} 
\vee \sqrt{\frac{t}{N}} \vee \frac{t}{N}\Biggr)\ge \frac{1}{2},
\]
and hence with probability at least $1-e^{-N/4C^2}-e^{-t}$ it holds that
\begin{align}\label{eq:local_aux2}
\left\|\Big(\big(\widehat{\Sigma}_{N_{\mathbf{j},2}}\big)^{-1}\Big)_{\mathbf{jj^{\prime}}}-\Big(\big(\Sigma_{N_{\mathbf{j},2}}\big)^{-1}\Big)_{\mathbf{jj^{\prime}}}\right\|&\le\|\Omega\|\left\|\left(\Sigma_{N_{\mathbf{j},2}}^{-1/2}\widehat{\Sigma}_{N_{\mathbf{j},2}}\Sigma_{N_{\mathbf{j},2}}^{-1/2}\right)^{-1}-I\right\| \nonumber \\
&\le \|\Omega\| \left\|\left(\Sigma_{N_{\mathbf{j},2}}^{-1/2}\widehat{\Sigma}_{N_{\mathbf{j},2}}\Sigma_{N_{\mathbf{j},2}}^{-1/2}\right)^{-1}\right\|\left\|\Sigma_{N_{\mathbf{j},2}}^{-1/2}\widehat{\Sigma}_{N_{\mathbf{j},2}}\Sigma_{N_{\mathbf{j},2}}^{-1/2}-I\right\|\nonumber\\
&\le 2C\|\Omega\|\Biggl(\sqrt{\frac{5^d b^d}{N}} \vee \frac{5^db^d}{N} 
\vee \sqrt{\frac{t}{N}} \vee \frac{t}{N}\Biggr)\nonumber\\
&=2C\|\Omega\|\Biggl(\sqrt{\frac{5^d b^d}{N}}
\vee \sqrt{\frac{t}{N}}\Biggr).
\end{align}

To bound the second term in \eqref{eq:local_aux1}, we consider the following partition of $\Sigma$ and $\Omega$
\[
\Sigma=\begin{pmatrix} \Sigma_{N_{\mathbf{j},2}} & \Sigma_{N_{\mathbf{j},2},N_{\mathbf{j},2}^{c}} \\ \Sigma_{N_{\mathbf{j},2}^c,N_{\mathbf{j},2}} & \Sigma_{N_{\mathbf{j},2}^{c}} \end{pmatrix},\quad \Omega=\begin{pmatrix} \Omega_{N_{\mathbf{j},2}} & \Omega_{N_{\mathbf{j},2},N_{\mathbf{j},2}^{c}} \\ \Omega_{N_{\mathbf{j},2}^c,N_{\mathbf{j},2}} & \Omega_{N_{\mathbf{j},2}^{c}} \end{pmatrix},
\]
where we use the shorthand notation $N_{\mathbf{j},2}^c:=([S]\times [S]\times \cdots \times [S]) \backslash  N_{\mathbf{j},2}$. Applying the formula for the inverse of block matrix (Lemma \ref{lemma:block_matrix_inverse}) leads to
\[
(\Sigma_{N_{\mathbf{j},2}})^{-1}=\Omega_{N_{\mathbf{j},2}}-\Omega_{N_{\mathbf{j},2},N_{\mathbf{j},2}^{c}}(\Omega_{N_{\mathbf{j},2}^{c}})^{-1}\Omega_{N_{\mathbf{j},2}^c,N_{\mathbf{j},2}}.
\]
Therefore,
\begin{align}\label{eq:local_aux6}
\left\|\Big(\big(\Sigma_{N_{\mathbf{j},2}}\big)^{-1}\Big)_{\mathbf{jj^{\prime}}}-\Omega_{\mathbf{jj^{\prime}}}\right\|&=\left\|\Big(\Omega_{N_{\mathbf{j},2}}-\Omega_{N_{\mathbf{j},2},N_{\mathbf{j},2}^{c}}(\Omega_{N_{\mathbf{j},2}^{c}})^{-1}\Omega_{N_{\mathbf{j},2}^c,N_{\mathbf{j},2}}\Big)_{\mathbf{jj^{\prime}}}-\Omega_{\mathbf{jj^{\prime}}}\right\| \nonumber\\
&=\left\|\Big(\Omega_{N_{\mathbf{j},2},N_{\mathbf{j},2}^{c}}(\Omega_{N_{\mathbf{j},2}^{c}})^{-1}\Omega_{N_{\mathbf{j},2}^c,N_{\mathbf{j},2}}\Big)_{\mathbf{jj^{\prime}}}\right\|\nonumber\\
&=\left\|\big(\Omega_{N_{\mathbf{j},2},N_{\mathbf{j},2}^{c}}\big)_{\mathbf{j},:}(\Omega_{N_{\mathbf{j},2}^{c}})^{-1}\big(\Omega_{N_{\mathbf{j},2}^c,N_{\mathbf{j},2}}\big)_{:,\mathbf{j^{\prime}}}\right\| \nonumber\\
&\le \left\|\big(\Omega_{N_{\mathbf{j},2},N_{\mathbf{j},2}^{c}}\big)_{\mathbf{j},:}\right\| \left\|(\Omega_{N_{\mathbf{j},2}^{c}})^{-1}\right\| \left\|\big(\Omega_{N_{\mathbf{j},2}^c,N_{\mathbf{j},2}}\big)_{:,\mathbf{j^{\prime}}}\right\| \nonumber\\
&\le \|\Omega^{-1}\| \left\|\big(\Omega_{N_{\mathbf{j},2},N_{\mathbf{j},2}^{c}}\big)_{\mathbf{j},:}\right\| \left\|\big(\Omega_{N_{\mathbf{j},2}^c,N_{\mathbf{j},2}}\big)_{:,\mathbf{j^{\prime}}}\right\|,
\end{align}
where the last step follows by $\left\|(\Omega_{N_{\mathbf{j},2}^{c}})^{-1}\right\|=1/\lambda_{\min}(\Omega_{N_{\mathbf{j},2}^{c}})\le 1/\lambda_{\min}(\Omega)=\|\Omega^{-1}\|$. Next, applying the matrix norm inequality $\|A\|^2\le \|A\|_{1}\|A\|_{\infty}$,
\begin{align}\label{eq:local_aux4}
\left\|\big(\Omega_{N_{\mathbf{j},2},N_{\mathbf{j},2}^{c}}\big)_{\mathbf{j},:}\right\|^2 \le \left\|\big(\Omega_{N_{\mathbf{j},2},N_{\mathbf{j},2}^{c}}\big)_{\mathbf{j},:}\right\|_1 \left\|\big(\Omega_{N_{\mathbf{j},2},N_{\mathbf{j},2}^{c}}\big)_{\mathbf{j},:}\right\|_{\infty}\le \bigg(\max_{t\in \mathcal{G}_{d,p}}\sum_{t^{\prime}:\|t-t^{\prime}\|_1\ge 2b} |\omega(t,t^{\prime})|\bigg)^2\le \|\Omega \|^2 e^{-4b},
\end{align}
\begin{align}\label{eq:local_aux5}
\left\|\big(\Omega_{N_{\mathbf{j},2}^c,N_{\mathbf{j},2}}\big)_{:,\mathbf{j^{\prime}}}\right\|^2 \le \left\|\big(\Omega_{N_{\mathbf{j},2}^c,N_{\mathbf{j},2}}\big)_{:,\mathbf{j^{\prime}}}\right\|_1 \left\|\big(\Omega_{N_{\mathbf{j},2}^c,N_{\mathbf{j},2}}\big)_{:,\mathbf{j^{\prime}}}\right\|_{\infty}\le \bigg(\max_{t\in \mathcal{G}_{d,p}}\sum_{t^{\prime}:\|t-t^{\prime}\|_1\ge b} |\omega(t,t^{\prime})|\bigg)^2\le \|\Omega \|^2 e^{-2b},
\end{align}
where the second inequality in \eqref{eq:local_aux4} follows since for any $\mathbf{j^{\prime\prime}}\in N_{\mathbf{j},2}^{c}$ it holds by definition that $\|\mathbf{j}-\mathbf{j^{\prime\prime}}\|_{\infty}\ge 3$, and hence only $\omega(t,t^{\prime})$ with $\|t-t^{\prime}\|_1\ge 2b$ appears in the rectangular matrix $\big(\Omega_{N_{\mathbf{j},2},N_{\mathbf{j},2}^{c}}\big)_{\mathbf{j},:}$; similarly, \eqref{eq:local_aux5} follows since for any $\mathbf{j^{\prime\prime}}\in N_{\mathbf{j},2}^{c}$, $\|\mathbf{j^{\prime}}-\mathbf{j^{\prime\prime}}\|_{\infty}\ge\|\mathbf{j}-\mathbf{j^{\prime\prime}}\|_{\infty}-\|\mathbf{j}-\mathbf{j^{\prime}}\|_{\infty}\ge  3-1=2$, and hence only $\omega(t,t^{\prime})$ with $\|t-t^{\prime}\|_1\ge b$ appears in the rectangular matrix $\big(\Omega_{N_{\mathbf{j},2}^c,N_{\mathbf{j},2}}\big)_{:,\mathbf{j^{\prime}}}$.
Then, it follows from \eqref{eq:local_aux6}, \eqref{eq:local_aux4}, and \eqref{eq:local_aux5} that
\begin{align}\label{eq:local_aux3}
\left\|\Big(\big(\Sigma_{N_{\mathbf{j},2}}\big)^{-1}\Big)_{\mathbf{jj^{\prime}}}-\Omega_{\mathbf{jj^{\prime}}}\right\|\le \|\Omega^{-1}\|\|\Omega\|e^{-2b}\|\Omega\|e^{-b}= \|\Omega\|\kappa(\Omega)e^{-3b}.
\end{align}
Combining \eqref{eq:local_aux1}, \eqref{eq:local_aux2}, and \eqref{eq:local_aux3} together gives that, for any $t\ge 1$, it holds with probability at least $1-e^{-N/4C^2}-e^{-t}$ that
 \[
 \frac{\|T_{\mathbf{j j^{\prime}}}-\Omega_{\mathbf{j j^{\prime}}}\|}{\|\Omega\|}\le 2C\Biggl(\sqrt{\frac{5^d b^d}{N}}  \vee \sqrt{\frac{t}{N}} \Biggr)+\kappa(\Omega)e^{-3b}.
 \]
\end{proof}

We are now ready to prove Theorem \ref{thm:mainresult3}.

\begin{proof}[Proof of Theorem \ref{thm:mainresult3}]
By definition of $\widehat{\Omega}$,
\[
\|\widehat{\Omega}-\Omega\|=\|\frac{1}{2}\big(\widetilde{\Omega}+\widetilde{\Omega}^{\top}\big)-\Omega\|\le \frac{1}{2}\|\widetilde{\Omega}-\Omega\|+\frac{1}{2}\|\widetilde{\Omega}^{\top}-\Omega\|=\|\widetilde{\Omega}-\Omega\|,
\]
where the last equality follows since $\Omega=\Omega^{\top}$. Now it suffices to bound $\|\widetilde{\Omega}-\Omega\|$.

    Define $\Omega_1=(\omega_1(t,t^{\prime}))_{t,t^{\prime}\in \mathcal{G}_{d,p}}$ such that $\omega_1(t,t^{\prime})=\omega(t,t^{\prime})$, if $t\in B_{\mathbf{j}}, t^{\prime}\in B_{\mathbf{j^{\prime}}}$ and $\|\mathbf{j}-\mathbf{j^{\prime}}\|_{\infty}\le 1$, and $0$ otherwise. Let $\Omega_2=\Omega-\Omega_1$. Then, 
    \begin{align*}
    \|\widetilde{\Omega}-\Omega\|\le    \|\widetilde{\Omega}-\Omega_1\|+\|\Omega_2\|.
    \end{align*}
    
   First, $\Omega \in \mathcal{F}_{d,p}$ implies that
   \[
   \|\Omega_2\|\le \|\Omega_2\|_{\ell_1\to \ell_1}\le \max_{t\in \mathcal{G}_{d,p}}\sum_{t^{\prime}:\|t-t^{\prime}\|_1\ge b} |\omega(t,t^{\prime})|\lesssim \|\Omega \| e^{-b}.
   \]

   To bound $\|\widetilde{\Omega}-\Omega_1\|$, we notice that for any $u,v\in\ell_2(\mathcal{G}_{d,p})$ with $\|u\|=\|v\|=1$,
   \begin{align*}
       |\langle u, (\widetilde{\Omega}-\Omega_1)v\rangle| &\le \sum_{\|\mathbf{j}-\mathbf{j^{\prime}}\|_{\infty}\le 1} |\langle u_{B_{\mathbf{j}}}, (T_{\mathbf{jj^{\prime}}}-\Omega_{\mathbf{jj^{\prime}}}) v_{B_{\mathbf{j^{\prime}}}}\rangle|\\
       &\le \sum_{\|\mathbf{j}-\mathbf{j^{\prime}}\|_{\infty}\le 1}\|u_{B_{\mathbf{j}}}\|\|v_{B_{\mathbf{j^{\prime}}}}\| \|T_{\mathbf{jj^{\prime}}}-\Omega_{\mathbf{jj^{\prime}}}\|\\
       &\le \left(\max_{\|\mathbf{j}-\mathbf{j^{\prime}}\|_{\infty}\le 1}\|T_{\mathbf{jj^{\prime}}}-\Omega_{\mathbf{jj^{\prime}}}\|\right)\times \Bigg(\sum_{\|\mathbf{j}-\mathbf{j^{\prime}}\|_{\infty}\le 1}\|u_{B_{\mathbf{j}}}\|\|v_{B_{\mathbf{j^{\prime}}}}\|\Bigg)\\
       &\le \left(\max_{\|\mathbf{j}-\mathbf{j^{\prime}}\|_{\infty}\le 1}\|T_{\mathbf{jj^{\prime}}}-\Omega_{\mathbf{jj^{\prime}}}\|\right)\times \Bigg(\frac{1}{2}\sum_{\|\mathbf{j}-\mathbf{j^{\prime}}\|_{\infty}\le 1}\Big(\|u_{B_{\mathbf{j}}}\|^2+\|v_{B_{\mathbf{j^{\prime}}}}\|^2\Big)\Bigg)\\
       &\le \left(\max_{\|\mathbf{j}-\mathbf{j^{\prime}}\|_{\infty}\le 1}\|T_{\mathbf{jj^{\prime}}}-\Omega_{\mathbf{jj^{\prime}}}\|\right)\times \frac{3^d}{2} \Bigg(\sum_{j\in [S]\times \cdots \times[S]} \|u_{B_{\mathbf{j}}}\|^2+\sum_{j\in [S]\times \cdots \times[S]} \|v_{B_{\mathbf{j}}}\|^2\Bigg)\\
       &=3^d \max_{\|\mathbf{j}-\mathbf{j^{\prime}}\|_{\infty}\le 1}\|T_{\mathbf{jj^{\prime}}}-\Omega_{\mathbf{jj^{\prime}}}\|,
   \end{align*}
   where we use the notation $a_B:=(a(t))_{t\in B}$ for any $a\in\ell_2(\mathcal{G}_{d,p})$, and the prefactor $3^d$ follows from the fact that $\mathrm{Card}(\{\mathbf{j^{\prime}}:\|\mathbf{j^{\prime}}-\mathbf{j}\|_{\infty}\le 1\})\le 3^d$ for any $\mathbf{j}\in [S]\times [S]\times \cdots \times [S]$. Therefore,
   \begin{align*}
       \|\widetilde{\Omega}-\Omega_1\|=\sup_{u,v\in\ell_2(\mathcal{G}_{d,p}):\|u\|=\|v\|=1} |\langle u, (\widetilde{\Omega}-\Omega_1)v \rangle|\le 3^d \max_{\|\mathbf{j}-\mathbf{j^{\prime}}\|_{\infty}\le 1}\|T_{\mathbf{jj^{\prime}}}-\Omega_{\mathbf{jj^{\prime}}}\|,
   \end{align*}
   and hence
   \begin{align}\label{eq:main3_aux1}
\|\widetilde{\Omega}-\Omega\| \le  \|\widetilde{\Omega}-\Omega_1\|+ \|\Omega_2\| \le 3^d  \left(\max_{\|\mathbf{j}-\mathbf{j^{\prime}}\|_{\infty}\le 1}\|T_{\mathbf{jj^{\prime}}}-\Omega_{\mathbf{jj^{\prime}}}\|\right)+\|\Omega\|e^{-b}.
   \end{align}

By Proposition \ref{prop:local_estimate}, if $N\ge 4C_0^2 5^d b^d$, then for any $t\ge 1$, it holds with probability at least $1-e^{-N/4C_0^2}-e^{-t}$ that
\begin{align}\label{eq:main3_aux2}
 \frac{\|T_{\mathbf{j j^{\prime}}}-\Omega_{\mathbf{j j^{\prime}}}\|}{\|\Omega\|}\le 2C_0\Biggl(\sqrt{\frac{5^d b^d}{N}} 
+ \sqrt{\frac{t}{N}} \Biggr)+\kappa(\Omega)e^{-3b},
 \end{align}
 where $C_0$ is an absolute constant. Given $r>0$, we set
\[
t_{*}:=2C_0\left(\sqrt{\frac{5^d b^d}{N}}+\sqrt{\frac{(r+1)\big(\log(S^d 3^d)+ (\log N)^d\big)}{N}}\right)+\kappa(\Omega)e^{-3b}.
\]
Then, taking a union bound and applying \eqref{eq:main3_aux2} leads to
\begin{align}\label{eq:main3_aux3}
    \mathbb{P}\left[\max_{\|\mathbf{j}-\mathbf{j^{\prime}}\|_{\infty}\le 1}\frac{\|T_{\mathbf{jj^{\prime}}}-\Omega_{\mathbf{jj^{\prime}}}\|}{\|\Omega\|}\ge t_{*}\right] 
    &\le \sum_{\|\mathbf{j}-\mathbf{j^{\prime}}\|_{\infty}\le 1}\mathbb{P}\left[\frac{\|T_{\mathbf{j j^{\prime}}}-\Omega_{\mathbf{j j^{\prime}}}\|}{\|\Omega\|}\ge t_{*}\right]  \nonumber\\
    &\le S^d 3^d \left(e^{-N/4C_0^2}+e^{-(r+1)\big(\log(S^d 3^d)+ (\log N)^d \big)}\right) \nonumber\\
    &\overset{\text{(i)}}{\le} 2 S^d 3^d e^{-(r+1)\big(\log(S^d 3^d)+ (\log N)^d \big)}  \nonumber\\
    &\le 2(S^d 3^d)^{-r} e^{-(r+1)(\log N)^d},
\end{align}
where (i) follows by assuming $N>4C_0^2 (r+1)\left(\log(S^d 3^d)+ (\log N)^d\right)$.

Combining \eqref{eq:main3_aux1} and \eqref{eq:main3_aux3} gives that, if $N>4C_0^2 \left(5^d b^d + (r+1)\big(\log(S^d 3^d)+(\log N)^d\big)\right)$, then
\begin{align*}
   \mathbb{P} \left[ \frac{\|\widetilde{\Omega}-\Omega\|}{\|\Omega\|}\ge 3^d t_* +e^{-b} \right]\le \mathbb{P}\left[\max_{\|\mathbf{j}-\mathbf{j^{\prime}}\|_{\infty}\le 1}\frac{\|T_{\mathbf{jj^{\prime}}}-\Omega_{\mathbf{jj^{\prime}}}\|}{\|\Omega\|}\ge t_{*}\right]\le 2(S^d 3^d)^{-r} e^{-(r+1)(\log N)^d}.
\end{align*}

Recall that $S=\lceil p/b\rceil$. If $p>\log (N\kappa(\Omega))$, by taking $b=\lceil \log (N \kappa(
\Omega
))\rceil$, we have
\begin{align*}
\mathbb{P} \left[ \frac{\|\widetilde{\Omega}-\Omega\|}{\|\Omega\|}\ge C_2\left(\sqrt{\frac{(\log N)^d}{N}}+\sqrt{\frac{\log p}{N}}+\sqrt{\frac{(\log \kappa(\Omega))^d}{N}}\ \right) \right]\le 2\left(\frac{\log (N\kappa(\Omega))}{p}\right)^{rd}e^{-(r+1)(\log N)^d},
\end{align*}
provided that $N\ge C_1\left((\log N)^d+\log p+(\log \kappa(\Omega))^d\right)$, where $C_1, C_2>0$ are two constants depending only on $r$ and $d$. Since $\|\widehat{\Omega}-\Omega\|\le \|\widetilde{\Omega}-\Omega\|$, the proof is complete.
\end{proof}

 
We conclude this section with a lemma that was used in the proof of Theorem \ref{thm:mainresult3}, which concerns the inverse of a block matrix.
\begin{lemma}[\cite{lu2002inverses}]\label{lemma:block_matrix_inverse}
For a positive definite matrix $\Sigma \in \mathbb{R}^{M \times M}$ in the block form
\[
\Sigma=\left(\begin{array}{ll}
\Sigma_{11} & \Sigma_{12} \\
\Sigma_{21} & \Sigma_{22}
\end{array}\right),
\]
it holds that 
\[ \Sigma^{-1} = 
\left(\begin{array}{cc}
\left(\Sigma_{11}-\Sigma_{12} \Sigma_{22}^{-1} \Sigma_{2 1}\right)^{-1} & -\Sigma_{1 1}^{-1} \Sigma_{1 2}\left(\Sigma_{2 2}-\Sigma_{2 1} \Sigma_{1 1}^{-1} \Sigma_{1 2}\right)^{-1} \\
-\Sigma_{2 2}^{-1} \Sigma_{2 1}\left(\Sigma_{11}-\Sigma_{1 2} \Sigma_{2 2}^{-1} \Sigma_{2 1}\right)^{-1} & \left(\Sigma_{22}-\Sigma_{21} \Sigma_{11}^{-1} \Sigma_{12}\right)^{-1}
\end{array}\right) .
\]
\end{lemma}

\section{Conclusions, discussion, and future directions}\label{sec:Conclusions}

This paper has studied estimation of large precision matrices and Cholesky factors obtained by observing a Gaussian process at many locations. We have shown that, under general assumptions on the precision operator and the observations, the sample complexity scales poly-logarithmically with the size of the target matrix. To address the challenge posed by the ill-conditioning of the precision matrix, we proposed an intuitive local regression technique on the lattice graph to exploit the approximate sparsity induced by the screening effect. For Cholesky factor estimation, we leveraged a block-Cholesky decomposition recently used to establish complexity bounds for sparse Cholesky factorization. 

Several questions arise from this work.
As mentioned in the discussion of Theorem \ref{thm:mainresult2}, the upper bound in Theorem \ref{thm:mainresult2} includes a logarithmic factor $\log(1/\mathsf{h})$ not present in Theorem \ref{thm:mainresult1}. While this additional factor can be avoided when estimating instead a modified block Cholesky factor (Remark \ref{rem:removinglogfactor}), a natural question is whether it can be eliminated directly in Theorem \ref{thm:mainresult2}. We conjecture that this could be achieved by integrating into our estimator a more efficient grouping procedure based on geometric information. The concept of \emph{supernode} from computational perspectives \cite{schäfer2021compression,schäfer2021sparse} may offer valuable insights in this direction.

In this paper, we focused on the case where data consist of Dirac-type measurements of Gaussian processes (i.e. pointwise observations), but our estimation procedure readily extends to other type of local measurements, such as local averages \cite[Example 4.4]{owhadi2019operator}. Throughout the paper, we assumed $s>d/2$ to ensure that the elements of $H_0^{s}(\D)$ are continuous and that pointwise evaluations of the Green’s function are well defined. For $s\le d/2$, our estimation framework can be adapted by replacing pointwise evaluations with local averages, accommodating cases where data consist of local averages of random fields.

Our estimation theory crucially relies on the exponential decay properties of the precision matrices established in \cite{owhadi2017multigrid,owhadi2017universal,owhadi2019operator} for integer $s$. For non-integer values of $s$ (e.g., when $\L$ is a fractional Laplacian), numerical experiments in 
\cite{schäfer2021compression} suggest that the exponential decay property still holds, see also \cite[Remark 24.1]{owhadi2019operator}. The recent work \cite{brown2018numerical} establishes localization results for fractional partial differential operators using the Caffarelli–Silvestre extension, supporting this hypothesis for $s=1/2$. These results immediately facilitate statistical estimation of precision matrices and Cholesky factors in such cases.

An important avenue for future research is to investigate precision operator estimation under continuous functional norms, beyond the operator norm on the lattice graph we considered. As noted in \cite{al2024optimal}, the infinite-dimensional operator estimation problem lends itself naturally to consider norms that account for the smoothness of Gaussian process data. Another direction for future work concerns fast simulation of Gaussian processes. While our focus has been on studying the sample complexity of estimating precision and Cholesky factors, the proposed local regression estimator is also expected to significantly enhance computational efficiency in downstream tasks, including sampling, transport, and other numerical simulations \cite{marzouk2016sampling,herrmann2020multilevel,harbrecht2021multilevel,kidd2022bayesian,katzfuss2024scalable}.

\section*{Acknowledgments}
This work was funded by NSF CAREER award NSF DMS-2237628.

\bibliographystyle{siamplain}
\bibliography{references}

\appendix

\section{Proofs in Section \ref{sec:proof_Cholesky}}\label{app:A}

\begin{proof}[Proof of Lemma \ref{lemma:perturbation_inverse}]
Note that
   \[
 \widetilde{\varepsilon}_{B}:=  \frac{\|(\widehat{B})^{-1}-B^{-1}\|}{\|B^{-1}\|}=\frac{\|(\widehat{B})^{-1}(\widehat{B}-B)B^{-1}\|}{\|B^{-1}\|}\le \|(\widehat{B})^{-1}\|\|\widehat{B}-B\|\le \|(\widehat{B})^{-1}\|\|B\|\varepsilon_B,
   \]
   and 
   \[
   \|(\widehat{B})^{-1}\|\le \|B^{-1}\|+\|(\widehat{B})^{-1}-B^{-1}\|= \|B^{-1}\|(1+\widetilde{\varepsilon}_B).
   \]
   Combining the two inequalities gives
   \[
   \widetilde{\varepsilon}_B\le \|B^{-1}\|(1+\widetilde{\varepsilon}_B)\|B\|\varepsilon_B=\kappa(B)\varepsilon_B(1+\widetilde{\varepsilon}_B),
   \]
   and therefore
   \[
   \widetilde{\varepsilon}_B\le \frac{\kappa(B)\varepsilon_B}{1-\kappa(B)\varepsilon_B}.
   \]
\end{proof}

\begin{proof}[Proof of Lemma \ref{lemma:pertubation_Cholesky}]
Let $\delta B:=\widehat{B}-B$ and $\delta L:=\widehat{L}-L$. Then,
\[
B+\delta B=(L+\delta L)(L+\delta L)^{\top}=L(I+W)L^{\top},
\]
where $W=L^{-1}(\delta B) L^{-\top}$ satisfies the ``normalized'' equation
\[
I+W=(I+L^{-1}\delta L) (I+L^{-1}\delta L)^{\top}.
\]
It follows that
\[
W=L^{-1}\delta L+(L^{-1}\delta L)^{\top} +(L^{-1}\delta L)(L^{-1}\delta L)^{\top},
\]
and $(L^{-1}\delta L)_{ii}=(L^{-1}\widehat{L}-I)_{ii}\ge -1$ for all $1\le i\le M$. By \cite[Theorem 3.1]{edelman1995parlett},
\[
\frac{\|L^{-1}\delta L\|}{\|W\|}=\frac{\|L^{-1}\delta L\|}{\|L^{-1}\delta L+(L^{-1}\delta L)^{\top} +(L^{-1}\delta L)(L^{-1}\delta L)^{\top}\|}< 2\log_2 M+4.
\]
Consequently,
\begin{align*}
\frac{\|\widehat{L}-L\|}{\|L\|}=\frac{\|\delta L\|}{\|L\|}\le \|L^{-1}\delta L\|&<(2\log_2 M+4)\|W\|\le (2\log_2 M+4) \|L^{-1}\|^2\|\delta B\|\\
&=(2\log_2 M+4) \|B^{-1}\|\|\delta B\|\le (2\log_2 M+4)\kappa(B)\varepsilon_B.
\end{align*}
\end{proof}

\end{document}